\documentclass[11pt,english,oneside]{article}
\usepackage{color}
\usepackage{mathtools}

\DeclarePairedDelimiter\ceil{\lceil}{\rceil}
\DeclarePairedDelimiter\floor{\lfloor}{\rfloor}

\makeatletter
\numberwithin{equation}{section} 
\numberwithin{figure}{section} 
  \@ifundefined{theoremstyle}{\usepackage{amsthm}}{}
  \theoremstyle{plain}
  \newtheorem{thm}{Theorem}[section]
  \theoremstyle{plain}
  \newtheorem{cor}[thm]{Corollary}
  \theoremstyle{plain}
  \newtheorem{prop}[thm]{Proposition}
  \theoremstyle{plain}
  \newtheorem{rem}[thm]{Remark}
  \theoremstyle{plain}
 \newtheorem{mydef}{Definition}
  
  \theoremstyle{plain}
  \newtheorem{lem}[thm]{Lemma}
  \theoremstyle{definition}

\begin{document}

\centerline{\Large \bf A small variation of the Taylor Method and }
\vskip.2cm
\centerline{\Large \bf periodic solutions of  the 3-body problem}
\vskip.3cm
\centerline{\large \bf Oscar Perdomo}

\title{A small variation of the Taylor Method and periodic solutions of the 3-body problem}



\begin{abstract}

In this paper we define a small variation of the Taylor method and a formula for the global error of this new numerical method that allows us to keep track of the round-off  error and does not require previous knowledge of the exact solution. As an application we provide a rigorous proof of the construction/existence of a periodic solution of the three body problem. Some images of this periodic motion can be seen at  {\color{blue} https://www.youtube.com/watch?v=fSmQyeKcj5k}
\end{abstract}



\begin{figure}[hbtp]
\begin{center}\includegraphics[width=.7\textwidth]{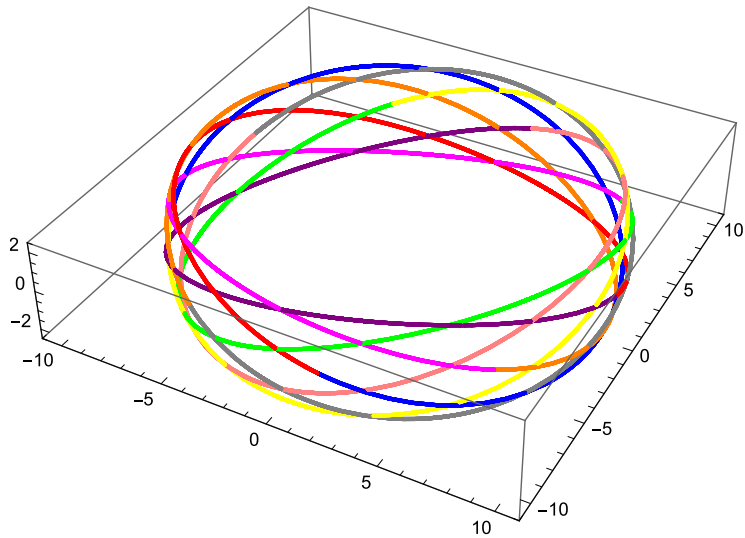}
\end{center}
\caption{This is the orbit of one of  the bodies in the periodic solution that we have mathematically shown to exist. The three masses are 100, 100 and 200 and the period of this solution is
$T=96.241\dots$ We are assuming that the gravitational constant is 1.}\label{fig1}
\end{figure}

\section{Introduction}

\subsection{The numerical Part} When we think about using a numerical method to do a mathematical proof, the first thing  that comes to our mind is that we need to consider the error of the numerical method. Very soon we realize that the standard formula for the error is no very useful due to the fact that it assumes that all the basic operations are being made with no error and this is computationally very expensive. As an example, if we consider the initial value problem  $y'=y-\frac{1}{3} y^2$ with $y(0)=\frac{1}{2}$ and we want to estimate $y(\frac{3}{10})$ using the Euler Method with $h=\frac{1}{100}$, we see that, even though we are considering the computational cheapest numerical method and we are only doing 30 iterations, a regular Computer Algebra System (Mathematica 10 in this case) will need 8696.99 seconds to do these $30$ iterations most likely because the final answer is a rational number of the form $\frac{p}{q}$ with p and  $q$ integers both with 2607760525 digits, more than 2.6 billions digits. On the other hand, if we allow the Computer Algebra System to have  round-off error in  every operation involved in each iteration, it becomes challenging to keep track of the error because easily, each iteration may have a few dozens of operations. In our example above, each iteration has $3$ operations: one raising to the square, one product and a difference. We will exploit the fact that most  of the Computer Algebra Systems  (CAS) can compute, with a mathematical precision, the two integers $i$ and $j$ such that $i\le r_1+\sqrt{r_2}\le j$ where $r_1$ and $r_2$ rational number and $r_2\ge0$. The new numerical method that we are proposing in this paper allows us to work all the time with mathematical precision without paying the price of dealing with numbers that have huge expression using integers. Our numerical method will
do the operations in each iteration with a mathematical precision and then at the end of the iteration, it will find with a mathematical precision a rational number that approximates the output of the iteration within a distance $H$. In order to be able to use the method, we  find a formula for  the error between the real value of the solution of the ODE and the approximation given by the numerical method in terms of the two values: $h$, the desired value for the  step, and $H$, the desired value for the rounding in each iteration. This is done in Theorem (\ref{error}). In general, when we  try to estimate the error of a numerical method, a problem that we face  is that we need to have some a-priori bounds of the solution of the differential equation that we do not know. An important aspect of Theorem (\ref{error}) is that it does not need previous knowledge of the exact solution.

\subsection{Periodic solutions of the three body problem} Poincare showed that the three body problem has a chaotic behavior that makes it difficult to solve. For this reason, it is not surprising that the only explicit solutions were discovered more than 240 years ago by Euler in 1765 and by Lagrange in 1772. 
Recall that we have a periodic solution when the values of the positions and the velocities of the three bodies, after some time $T$,  agree with the values of the positions and velocities at $t=0$. Usually, when the values of a solution after $T$ are within a small distance $\epsilon$ from the initial condition, this solution is called {\it numerically periodic}. There is an enormous amount of numerically periodic solutions, for example, in 1975 \cite{H}, H. Henon showed a family of numerically periodic planar solutions of the three body with $\epsilon=10^{-14}$.   Despite  the abundance  of numerically periodic solutions,  the task of showing that numerically periodic solutions are periodic is a difficult one. In  2000, Chenciner and Montgomery \cite{CM}  showed that a numerically solution found by Moore in 1993, \cite{Mo}, was indeed periodic. This example represents (to my knowledge) the first example of a periodic solution that has a numerical image associated with it and does not have an explicit formula. It is important to mention that proofs showing the existence of periodic solutions have been found before, for example, Meyer and Schmidt1993 \cite{M}. These solutions are usually near a bifurcation point of a family of solutions described with explicit formulas.
After the paper by Chenciner and Montgomery, there has been more proofs showing that some numerical solutions are periodic. Among them, we have  papers by Terracini and Ferrario, K.C Chen and Simo and Kapitza and Gronchi. An excellent account of the work done in this direction so far can be found in the site {\color{blue}http://montgomery.math.ucsc.edu/Nbdy.html}

In this paper we will show that the numerical periodic solution given by the initial condition explained in Figure \ref{fig2} is periodic.

\begin{figure}[hbtp]
\begin{center}\includegraphics[width=.7\textwidth]{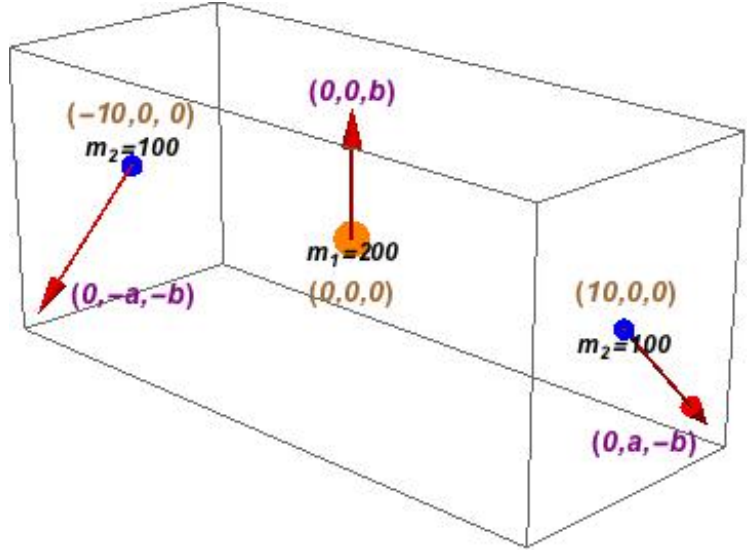}
\end{center}
\caption{There is a periodic solution with period $T=96.241\dots$, $a=4.3170\dots$ and $b=1.4903\dots$. }\label{fig2}
\end{figure}

To describe the motion, let us assume that  $t$ represents seconds. Keep in mind that the units of time, mass and distance, have been adjusted so that the gravitational constant is 1. For this periodic solution the motion starts with two bodies, each one with a mass of 100, separated  20 units and the third body with a mass of 200 right in the middle. If we reference the body in the middle as body 1, then, we have that the initial position of the  body 1 is $(0,0,0)$ and the initial position of the other two bodies are $(10,0,0)$ and $(-10,0,0)$. Body 1 will always stay in the $z$-axis while the other two bodies will move around the $z$-axis. During the first  $t_0=2.673\dots$ seconds, Body 1  moves straight up to the position $(0,0,2.453\dots)=(0,0,z_1)$, this is the farthest up this body will get. Simultaneously, the other two bodies move down and around the $z$-axis;  after these first $t_0$ seconds they are both $z_1$ units below the $x$-$y$ plane, they are $r_1=9.43\dots $ units away from the the $z$ axis; and they have made a $\frac{7\pi}{18}$-rotation (70 degrees) with respect to the $z$-axis. At every instance, the distance to the $z$-axis of these two bodies is the same and it is twice the distance between them.  In this way, after the first $t_0$ seconds Body  2 is at $(r_1 \cos \frac{7 \pi}{18},r_1 \sin\frac{7 \pi}{18},-z_1)$ and Body 3 is at $(-r_1 \cos \frac{7 \pi}{18},-r_1 \sin\frac{7 \pi}{18},-z_1)$. During the next $t_0$ seconds the body 1 moves down getting back to the origin and the other two bodies will simultaneously rotate, move up and move apart from the $z$ axis; they will rotate another $\frac{7 \pi}{18}$, and they will go back to be 10 units apart from the $z$-axis. Notice that after $2t_0$ seconds the relative positions with respect to each other are the same as in the starting position. The only difference is that Body 1 is now going down and the other two bodies are going up. With respect to a fixed reference frame, after $2t_0$ seconds, the positions of all three bodies differ from the starting position by a rotation of $\frac{7 \pi}{9}$. The next $t_0$ seconds Body 1 will reach its lowest point $(0,0,-z_1)$ while the other two bodies will reach  the highest point rotating another $\frac{7 \pi}{18}$. Finally after another $t_0$ seconds the relative positions with respect to each other are the same as the starting position, and this time the body 1 is going up as it was at the starting position. With respect to a fixed frame, the positions after $4 t_0$ seconds differ from the starting position by a rotation of $\frac{14 \pi}{9}$. Doing 9 more of these $4t_0$ cycles will bring the three bodies to the starting position. In this process, bodies 2 and 3 will have completed 7 rotations around the $z$-axis. The trajectory of Body 2 is shown in Figure \ref{fig1}. Each color represents a cycle of $4t_0$ seconds, for this reason there are 9 colors in the picture. Body 3 does not share the same trajectory as Body 2. The trajectory of Body 3 is the reflection with respect to the $x$-$y$ plane of the trajectory of the body 2.
  
We would like to emphasize that the goal of this part of the paper is not to show the existence of this type of solutions: This was done by Meyer and Schmidt \cite{M}, nor to show numerically periodic solutions  of this type: We can see similar images of solutions like the one we are showing (when the three masses are the same), in the work by Yan and Ouyang \cite{YO}. The goal of this part of the paper is to give a rigorous proof that a particular numerically periodic solution is indeed periodic. 

We will reduce the proof of  the periodicity of this solution to show that three functions defined in an open set of $R^3$  must vanish simultaneously. The variables in the domain of these three functions are given by triples $(t,a,b)$ where $a$ and $b$ are explained in Figure \ref{fig2}


The author has found some numerical solutions of the same type, not only for the 3 body problem but for the $n$-body problem. Some aspects of these solutions have been posted online:

The link {\color{blue} https://www.youtube.com/watch?v=PtEMb6Rvflg} shows a periodic solution of the 6 body problem.

The link {\color{blue}https://www.youtube.com/watch?v=2Wpv6vpOxXk} shows a periodic solution of the four body problem. 

The link {\color{blue}https://www.youtube.com/watch?v=hjQp1P09560} shows a periodic solution of the three body problem. 

The images in the videos were generated by solving the differential equation that governs the n body problem  and then posting 10 pictures per second, in this way, the time in the video is proportional to the time $t$ of the solution.

The technique used in this paper to prove periodicity is different to the one used by Chenciner and Montgomery, they used variational methods. There are five ingredients in our proof that the solution that we are considering is periodic: (i) The Poincare-Miranda theorem, Theorem \ref{PMT}, which is essentially a generalization of the intermediate value theorem.   (ii) A symmetry  result that allow to integrate the ordinary differential equation (ODE) over a quarter of  a period instead of the whole period.  Lemma \ref{symlem}. (iii) The Round Taylor Method to solve differential equations, that allow us to estimate the values of the functions that we are considering. Section \ref{nm}. (iv) A lemma related with the the implicit function theorem that allow us to find a set where the solution of a system of equations of the form $f_1(t,a,b)=0, f_2(t,a,b)=0$ is given  by only  points on a connected smooth curve, Section \ref{ift}. (v) A theorem that tell us how to compute the partial derivative of an ODE with respect to the initial conditions and parameters in the ODE. \cite{G}.

The author would like to express his gratitude to Andr\'es Rivera, Richard Montgomery, Carles Simo and Marian Anton for their valuable comments on this work.

\section{The Round Taylor Method}\label{nm}


\subsection{The Round Taylor Method} We will be using a small modification of the Taylor method of order $m$ to estimate solutions of differential equations. This modification will allow the computer that implements the method to  work  all the time with rational number and its square roots instead of working with approximations of numbers. In this way, the method is free of error coming from decimal approximation, this is, this method is {\it round -off } error free.  Let us call this method the {\it Round Taylor Method}.
Assume we have the  ordinary differential equation 
\begin{eqnarray}\label{ns}
Y^\prime(t)=f(Y(t)), \quad  Y(0)=y_0
\end{eqnarray}
where $Y(t)=(Y_1(t),\dots,Y_n(t))^T$ and
$ f(y)=(f_1(y),\dots,f_n(y))^T $. Here $v^T$ stands for the transpose of the vector $v$.  Given two positive rational number $h$ and $H$, we define the sequence of points in $R^n$ that starts with $z_0=y_0$ and follows using the recursive formulas 

\begin{eqnarray*}\label{tm}
y_{1}&=&z_0+f(z_0)h +F_1(z_0) \frac{h^2}{2!} +\dots F_{m-1}(z_0) \frac{h^m}{m!}\\
z_1&=&\quad\hbox{a rational number such that }\quad |z_1-y_1|\le H\\
y_2&=&z_1+f(z_1)h +F_1(z_1) \frac{h^2}{2!} +\dots F_{m-1}(z_1) \frac{h^m}{m!} \\
&\vdots&\\
z_i&=&\quad\hbox{a rational number such that }\quad |z_i-y_i|\le H\\
y_{i+1}&=&z_i+f(z_i)h +F_1(z_i) \frac{h^2}{2!} +\dots F_{m-1}(z_i) \frac{h^m}{m!}  \\
\end{eqnarray*}


where $F_1=Df f,\, F_2=DF_1 f, \dots \, F_{m-1}=DF_{m-2} f$. Here $D$ denotes the derivative operator that takes a function from $R^n$ to $R^n$ to the matrix which columns are the partial derivatives with respect to the variables. In order to define $z_i$ from $y_i$, we may use -and will use in this paper- the Floor function $\floor*{x}$ that assigns to a real number $x$ the largest integer not greater than $x$. When $x=(x_1,\dots,x_k)$ is a vector, $\floor*{x}=(\floor{x_1},\dots,\floor*{x_k})$. It is clear that a posible choice for $z_i$ is $z_i=H\floor*{\frac{1}{H} y_i}$. 

 As an example, when we use the Round Taylor method of order 1 for the differential equation given in the introduction, this is, $n=1$, $f(y)=y-\frac{y^2}{3}$, $y_0=\frac{1}{2}$. When $h=1/100$ and $H=10^{-6}$, we obtain that $y_0=1/2=z_0$ and  the next 10 values of $\{y_i,z_i\}$  are given by

$$\left(
\begin{array}{cc}
 \frac{121}{240} & \frac{252083}{500000} \\ \\
 \frac{38127028661111}{75000000000000} & \frac{12709}{25000} \\ \\
 \frac{96109156319}{187500000000} & \frac{256291}{500000} \\ \\
 \frac{38762401423319}{75000000000000} & \frac{16151}{31250} \\ \\
 \frac{152668926449}{292968750000} & \frac{521109}{1000000} \\ \\
 \frac{52541490803373}{100000000000000} & \frac{262707}{500000} \\ \\
 \frac{13243698510717}{25000000000000} & \frac{529747}{1000000} \\ \\
 \frac{160232709115991}{300000000000000} & \frac{534109}{1000000} \\ \\
 \frac{161549754576119}{300000000000000} & \frac{538499}{1000000} \\ \\
 \frac{162875215826999}{300000000000000} & \frac{542917}{1000000} 
\end{array}
\right) $$

Let us start the process of finding the formula for the error of this numerical method.

\begin{lem}\label{tsq} The sequence given by the recursive formula $q_{i+1}=(1+L)q_i+p$  (with $L\ne0$) satisfies that $q_k=\frac{p+Lq_0}{L}(1+L)^k-\frac{p}{L}$.
\end{lem}

\begin{proof}
The proof follows by induction.  Clearly the formula works for $k=0$. Now,  assuming the formula works for $k\ge 0$, we have:

$$q_{k+1}=(1+L)q_k+p=(1+L)\left(\frac{p+Lq_0}{L}(1+L)^k-\frac{p}{L}\right)+p=\frac{p+Lq_0}{L}(1+L)^{k+1}-\frac{p}{L}$$

This finishes the proof of the lemma. \end{proof}

There are different norms that we can use in the set of matrices, in order to establish the the one that we are using, we state the following lemma.

\begin{lem}
Let $A$ be an  $n\times n$ matrix,  if $|A|^2$ denote the sum of the square of its entries, then for any vector $v$ in $R^n$ we have that $|Av|\le |A||v|$.
\end{lem}
\begin{proof}
If $A^i$ denote the columns of $A$ and $v_i$ denote the entries of $v$, then we have that 

$$|Av|=|\sum_{i=1}^nv_iA^i|\le \sum_{i=1}^n|v_i||A^i|\le |A||v|$$
\end{proof}

\begin{thm}\label{error} Let $h$ and $H$ be two positive numbers and $k$ a positive integer.  Let us consider the sequences $\{(y_i,z_i)\}_{i=0}^k$ given by the Round Taylor Method associated with the values $h$ and $H$ for  the ODE $ Y^\prime(t)=f(Y(t)), \quad  Y(0)=y_0$, described in the beginning of this section.  Let us assume that we can find  constants $M_0,M_1,\dots M_n$, $K_0,\dots, K_{m-1}$  and sets 
$U_1=\{(u_1,\dots,u_n):b_1\le u_1\le c_1, \dots , b_n\le u_n\le c_n\} $ and $U_2=\{(u_1,\dots,u_n):b_1-\epsilon< u_1< c_1+\epsilon, \dots , b_n-\epsilon< u_n< c_n+\epsilon\} $ with 
$\epsilon>M_0 h+\tilde{H}$, where  $\tilde{H}=\frac{M \frac{h^m}{(m+1)!}+\frac{H}{h}}{L}\left(\hbox{e}^{Lkh}-1\right)$ and  $M=\sqrt{M_1^2+\dots +M_n^2}$ and $L=K_0+K_1 \frac{h}{2!}+\dots+K_{m-1} \frac{h^{m-1}}{m!}$. If 

\begin{itemize}
\item The map $f$ and all its partial derivative with order less than $m+2$ are continuous in an open set that contain the closure of $U_2$. 
\item $z_j \in U_1$ for $j=0,\dots , k$. 
\item $|{F_m}_i(u)|\le M_i$ for $i=1,\dots, n$,  for all $u\in U_2$. Here $F_m=({F_m}_1,\dots, {F_m}_n)$.
\item $|Df(u)|\le K_0$ and $|DF_i(u)|\le K_i$ for $i=1,\dots, m-1$ and $u\in U_2$.
\item $|f_i(u)|\le M_0$  for $i=1,\dots, n$ and $u\in U_2$.
\end{itemize}

then, the solution $Y(t)$ of the system of ordinary differential equation is defined on $[0,h k]$ and  for any positive integer $j\le k$, we have that $|Y(jh)-z_j|\le \tilde{H}$.
\end{thm}

\begin{proof}
Let us start by checking that for any pair of points $u_1$ and $u_2$ in $U_2$ we have that  $|f(u_2)-f(u_1)|\le K_0 |u_2-u_1|$ and $|F_i(u_2)-F_i(u_1)|\le K_i |u_2-u_1|$ for $i=1,\dots m-1$. We have that 

\begin{eqnarray*}
|f(u_2)-f(u_1)|&=&\left|\int_0^1\frac{df(u_1+t(u_2-u_1))}{dt}\, dt \right|\\
  &=&\left|\int_0^1 Df(u_1+t(u_2-u_1))\,  (u_2-u_1) \, dt \right|\le K_0 |u_2-u_1|
\end{eqnarray*}

The proof of the  inequalities $|F_i(u_2)-F_i(u_1)|\le K_i |u_2-u_1|$, for $i=1,\dots,m-1$ is similar. For any non negative integer $l\le  k$  such that $lh$ is in the domain of $Y$ let us denote by $x_l=Y(lh)$. Notice that anytime $Y(t)$ is  in the closure of $U_2$ for all $t\in[0,t_1]$, then it is possible to extend $Y(t)$ to an interval of the form $[0,t_1+\delta]$ with $\delta>0$. We 
will show that for any $l\le k$, $Y(t)$ is defined on $[0,lh]$, $Y(t)\in U_2$ for all $t\in[0,hl]$ and  $|Y(lh)-z_l| \le \tilde{H}$. Clearly the result hold for $l=0$ because $|x_0-z_0|=0$. Let us assume that $Y(t)$ is defined for all $t\in[0,jh]$ with  $Y(t)\in U_2$ and $|x_l-z_l|<\tilde{H}$ for all $0\le l\le j<k$, we will prove the theorem by showing that $Y(t)$ is defined for all $t\in [0,(j+1)h]$ with $Y(t)\in U_2$ and $|x_{j+1}-z_{j+1}| < \tilde{H}$.  Let us show that  for any $t\in [jh,(j+1)h]$, $Y(t)\in U_2$.  Since $Y(t)$ is continuous and $U_2$ is open, then $Y(hj+\tau)\in U_2$ for small positive values of $\tau$. Let us show by contraction that $Y(hj+\tau)\in U_2$ for all $\tau\in [0,h]$. If $Y(hj+\tau)\notin U_2$ for all $\tau\in [0,h]$, the we can find $\tau^\star$ such that $Y(jh+\tau)\in U_2$ for all $\tau\in (0,\tau^\star)$ and $Y(jh+\tau^\star)\notin U_2$. Therefore, writing $Y=(Y_1,\dots, Y_n)$, we can find a positive integer $w\le n$ such that either  $Y_w(jh+\tau^\star)=b_w-\epsilon$ or   $Y_w(jh+\tau^\star)=c_w+\epsilon$ and $Y_e(jh+\tau)\in (b_e-\epsilon,c_e+\epsilon)$ for all positive integers $e\le n$ and all $\tau\in(0,\tau^\star)$. Denoting $z_j=({z_j}_1,\dots,{z_j}_n)$, we have that for some $\xi\in (0,\tau^\star)$,

\begin{eqnarray*}
|Y_w(jh+\tau^\star)-{z_j}_w|&\le&|Y_w(jh+\tau^\star)-Y_w(jh)|+|Y_w(jh)-{z_j}_w|\\
     &\le&|\dot{Y}_w(hj+\xi)| \tau^\star + \tilde{H}  \\
 &\le&   |f_w(Y(jh+\xi))|\tau^\star+\tilde{H}\\
 &<&  M_0 h+\tilde{H}<\epsilon
\end{eqnarray*}

Since $z_j\in U_1$, the inequality above contradicts the fact that either  $Y_w(jh+\tau^\star)=b_w-\epsilon$ or   $Y_w(jh+\tau^\star)=c_w+\epsilon$. This contradiction shows that $Y(t)\in U_2$ for all $t\in [0,(j+1) h]$, in particular $x_{j+1}\in U_2$. Let us prove now that $|x_{j+1}-z_{j+1}|<\tilde{H}$.

 For any $i=0,\dots, j+1$ we have that 
\begin{eqnarray*}
|x_{i+1}-z_{i+1}|&\le& |x_{i+1}-y_{i+1}|+|y_{i+1}-z_{i+1} |\\
&\le&|x_i+f(x_i)h+F_1(x_i)\frac{h^2}{2!}+\dots+F_{m-1}(x_i)\frac{h^m}{m!}-y_{i+1}|+\frac{Mh^{m+1}}{(m+1)!}+H\\
&\le& |x_i-z_i| \left(  1+ h L\right)+\frac{Mh^{m+1}}{(m+1)!}+H
\end{eqnarray*}

Let us define $p=\frac{Mh^{m+1}}{(m+1)!}+H$, $q_0=0$ and $q_{i+1}=(1+hL)q_i+p$. By induction we can show that $|x_i-z_i|\le q_i$ for all $i\le j+1$. Using Lemma (\ref{tsq}) we obtain that 

\begin{eqnarray*}
|Y((j+1)h)-z_{j+1}|&=&|x_{j+1}-z_{j+1}|\le  \frac{p}{hL}\left((1+hL)^{j+1}-1\right)\\
    &\le & \frac{p}{hL}(\hbox{e}^{Lhk}-1)= \tilde{H}
\end{eqnarray*}

This finishes the proof.

\end{proof}

The following theorem is well known. A reference for the particular case when $f$ and $g$ are Lipschitz real value functions can be found at Earl A. Coddington, An Introduction to ordinary differential equations -Dover Publication - 1989. A reference for a more general case can be found at Herbert Amann, Ordinary Differential Equations An Introduction to Nonlinear Analysis - Walter de Gruyter - 1990.

\begin{thm}\label{error2} Let us assume that $f,g:U\subset R^n\to R^n$ are $C^1$ functions defined on an open convex set $U$ such that $|Df(u)|<K$, $|Dg(u)|<K$ and $|f(u)-g(u)|<\epsilon$. If $y(t)$ and $z(t)$ satisfy
$\dot{y}(t)=f(y(t))$ and $\dot{z}(t)=g(z(t))$ respectively, then

$$|y(t)-z(t)| \le |y(t_0)-z(t_0)|\hbox{e}^{K|t-t_0|}+\frac{\epsilon}{K}\left(\hbox{e}^{K|t-t_0|}-1 \right) $$

\end{thm}

\section{On the implicit function theorem}\label{ift}

Let us consider the set  $\Sigma=\{x=(x_1,x_2,x_3)\, :\, f_1(x)=0,\quad  f_2(x)=0 \}$ where $f_1$ and $f_2$ are real value functions defined on an open set of $R^3$. By the Implicit Function Theorem we know that if $p_0\in \Sigma$ and $\nabla f_1(p_0)\times \nabla f_2(p_0)$ is not the zero vector then there exists an open set $U$ that contains $p_0$ such that $\Sigma\cap U$ is given a regular connected curve. The following lemma give us an estimate on how big this open set $U$ can be, under the assumption that we know that there is point in $\Sigma$ in a small box with dimensions $2\mu_1$, $2\mu_2$ and $2\mu_3$.

\begin{thm}\label{iftt}
Let us assume that  $f_1,f_2:R^3\to R$ are smooth functions, $\mu_1$, $\mu_2$, $\mu_3$, $\delta_1$, $\epsilon_1$, $\epsilon_2$, $\tilde{\epsilon}_1$, $\tilde{\epsilon}_2$, $\delta_2$, $\tilde{\delta}_1$, $\tilde{\delta}_2$ are positive numbers such that for $i=1,2$, $\mu_i<\epsilon_i<\delta_i$, $\rho_i=\frac{\tilde{\epsilon_i}-\mu_i}{m_i}-\mu_3>0$ where $m_1=\frac{ \epsilon_1 (\tilde{\delta}_2+\epsilon_2)}{\delta_1 \delta_2-\epsilon_1\epsilon_2} $ and $m_2=\frac{ \epsilon_2 (\tilde{\delta}_1+\epsilon_1)}{\delta_1 \delta_2-\epsilon_1\epsilon_2} $ and

$$U=\{(x,y,z): |x|<\tilde{\epsilon}_1,\, |y|<\tilde{\epsilon}_2,\,   |z|<\rho  \} \,\hbox{where}\, \rho< \hbox{min}\{\rho_1,\rho_2 \}\, .$$

If  there exists $p_0=(x_0,y_0,z_0)$ such that $f_1(p_0)=f_2(p_0)=0$ with $|x_0|<\mu_1$, $|y_0|<\mu_2$, $|z_0|<\mu_3$ and
 for every $p\in U$ we have that

$$\delta_1<|\frac{\partial f_1}{\partial x_1}|<\tilde{\delta}_1,\, |\frac{\partial f_1}{\partial x_2}|<\epsilon_1,\, |\frac{\partial f_1}{\partial x_3}|<\epsilon_1,\, \delta_2<|\frac{\partial f_2}{\partial x_2}|<\tilde{\delta}_2,\, |\frac{\partial f_2}{\partial x_1}|<\epsilon_2,\, |\frac{\partial f_2}{\partial x_3}|<\epsilon_2$$

then, for every $z_1$ with $|z_1|<\rho$, there exists a unique $(x_1,y_1)$ such that  $(x_1,y_1,z_1)\in U$ is a solution of the equations $f_1(x,y,z)=0,\, f_2(x,y,z)=0$.
\end{thm}

\begin{proof}
Let us denote by $\xi=(\xi_1,\xi_2,\xi_3)=\nabla f_1\times \nabla f_2$. Direct computations show that $|\xi_1|\le \epsilon_1(\tilde{\delta}_2+\epsilon_2)$, $|\xi_2|\le \epsilon_2(\tilde{\delta}_1+\epsilon_1)$,  and $|\xi_3|\ge \delta_1 \delta_2-\epsilon_1 \epsilon_2$. Let $\alpha(t)=(x(t),y(t),z(t))$ be the integral curve of the vector field $\xi$ such that $\alpha(0)=p_0$. Since $|\xi_3|>0$ we have that the vector field $\xi$ never vanishes on $U$, therefore there exist $T_2>0$ and $T_1<0$  such that $\alpha(t)\in U$ for all $t\in (T_1,T_2)$ and $\alpha(T_1)$ and $\alpha(T_2)$ are in the boundary of $U$, in particular we have that either $x(T_2)=\pm \tilde{\epsilon}_1$, $y(T_2)=\pm \tilde{\epsilon}_2$ or  $z(T_2)=\pm \rho$. We will prove that $z(T_2)$ must be either $\rho$ or $-\rho$. Using the fact that $\dot{x}=\xi_1(\alpha(t))$,  $\dot{y}=\xi_2(\alpha(t))$ and  $\dot{z}=\xi_3(\alpha(t))$ we obtain that,

\begin{eqnarray}\label{slope}
|\frac{\dot{x}}{\dot{z}}|\le \frac{ \epsilon_1 (\tilde{\delta}_2+\epsilon_2)}{\delta_1 \delta_2-\epsilon_1\epsilon_2}=m_1\quad \hbox{and}\quad
 |\frac{\dot{y}}{\dot{z}}|\le \frac{ \epsilon_2 (\tilde{\delta}_1+\epsilon_1)}{\delta_1 \delta_2-\epsilon_1\epsilon_2}=m_2
\end{eqnarray} 

Recall that  $z(t)$ is one to one because its derivative never vanishes. Let us denote by $h$ the inverse function of $z$. We have that $h(z(t))=t$.  If $g_1(\tau)=x(h(\tau))$, then  $g_1(z_0)=x_0$ and $|g_1^\prime(\tau)|<\frac{ \epsilon_1 (\tilde{\delta}_2+\epsilon_2)}{\delta_1 \delta_2-\epsilon_1\epsilon_2}$. Therefore,

$$|x(T_2)-x_0|=|g(z(T_2))-g_1(z_0)|=|g^\prime(\tau^\star)||z(T_2)-z_0|  \le m_1|z(T_2)|+m_1\mu_3$$

and then $|x(T_2)|<\mu_1+m_1\rho_1+m_1\mu_3=\tilde{\epsilon}_1$. Likewise we can show that $|y(T_2)|<\tilde{\epsilon}_2$. Therefore we must have $|z(T_2)|=\rho$. The same arguments show that $|z(T_1)|$ must be either $\rho$ or $-\rho$. Since the function $\dot{z}(t)\ne0$ and $[T1,T_2]$ is connected then we get that for every $z_1$ with $|z_1|<\rho$, there exists a $p_1=(x_1,y_1,z_1)=\alpha(t_1)\in U$ such that $f_1( p_1)=0$ and $f_2(p_1)=0$. Let us prove that  $p_1$ is unique. If $p_2=(x_2,y_2,z_1)\in U$ and $p_2\ne p_1$, then either $|x_2-x_1|\le |y_2-y_1|\ne0$ or $|y_2-y_1|\le |x_2-x_1|\ne0$. If $|x_2-x_1|\le |y_2-y_1|\ne0$. Since we are assuming that $f_2(x_1,x_2,x_3)=0$ then, 

\begin{eqnarray*}
|f_2(x_2,y_2,z_1)|&\ge&  |f_2(x_2,y_2,z_1)-f_2(x_2,y_1,z_1)|- |f_2(x_2,y_1,z_1)-f_2(x_1,y_1,z_1)|\\
   &\ge&  |y_2-y_1|\delta_2-|x_2-x_1| \epsilon_2 >0
\end{eqnarray*}   

Likewise we can show that if $|y_2-y_1|\le |x_2-x_1|\ne0$, then $|f_1(p_2)|>0$. Therefore the only solution of the system $f_1(p)=f_2(p)=0$ on the set $U$ are those points in the curve $\alpha:[T_1,T_2]\to U$. This finishes the proof.
\end{proof}

For any $u=(u_1,u_2,u_3)$ and $v=(v_1,v_2,v_3)$, let us denote by $u\cdot v=u_1v_1+u_2v_2+u_3v_3$ the dot product of $u$ and $v$. Given a function $f:R^3\to R$, we  denote  $D_vf_i=\nabla f_i\cdot v$.  By doing an orthogonal change of coordinates we have the following corollary,

\section{The differential equation, a symmetry result and the Poincare Miranda Theorem}\label{sm}

Let us start introducing the differential equation for the subfamily of solutions of the three body that we are considering in this paper.

\begin{prop} If $f$, $r$ and $\theta$  satisfy the initial value system of differential equations

\begin{eqnarray}\label{e1}
\ddot{f}=-\frac{400}{s^3}f,\, \quad \ddot{r}=\frac{100 a^2}{r^3}-\frac{25}{ r^2}-\frac{200 r}{s^3},
\, r^2\dot{\theta}=10 a,
\end{eqnarray}
with $\, r(0)=10, \, f(0)=0,\, \dot{r}(0)=0,\, \dot{f}(0)=b,\, \theta(0)=0$, where  $s=\sqrt{r^2+4f^2}$ 

then, 

\begin{eqnarray*}
x(t)&=&(0,0,f(t))\\
y(t)&=&(r(t) \cos(\theta(t)),r(t) \sin(\theta(t)),- f(t))\\
z(t)&=&(-r(t) \cos(\theta(t)),-r(t) \sin(\theta(t)),-f(t))\, ,
\end{eqnarray*}

is a solution of the 3-Body problem with the mass for the body moving according to $x(t)$ equal $200$ and the masses of the other two bodies equal $100$. We are assuming that the gravitational constant is 1.
\end{prop}

\begin{proof}  Notice that $|x-y|=|x-z|=h$ and $|y-x|=2r$, we can check that  if $f$ and $r$ satisfy the ODE, then $\ddot{x}(t)=100\frac{y-x}{|y-x|^3}+100\frac{z-x}{|z-x|^3}$,  $\ddot{y}(t)=200\frac{x-y}{|x-y|^3}+100\frac{z-y}{|z-y|^3}$ and $\ddot{z}(t)=200\frac{x-z}{|x-z|^3}+100\frac{y-z}{|y-z|^3}$.  For more details see \cite{P} or \cite{M} for example.
\end{proof}

\begin{lem}\label{per}
Let us assume that $f$ and $g$ satisfies the ordinary differential equation 
\begin{eqnarray*}
\ddot{f}&=&\phi(f,g)\\
\ddot{g}&=&\xi(f,g)
\end{eqnarray*}

with $f(0)=0$, $\dot{f}(0)=a$, $g(0)=r>0$ and $\dot{g}(0)=0$ and $\phi$ and $\xi$ smooth functions. If $\phi(-f,g)=-\phi(f,g)$  and $\xi(-f,g)=\xi(f,g)$  then, $f(t)$ is odd and $g(t)$ is even. 
\end{lem}

\begin{proof} Let us consider the functions $\tilde{f}(t)=-f(-t)$ and $\tilde{g}(t)=g(-t)$. A direct computation shows that $\tilde{f}(0)=0$, $\dot{\tilde{f}}(0)=a$, $\tilde{g}(0)=r$ and  $\dot{\tilde{g}}(0)=0$. Moreover we have that 

$$\ddot{\tilde{f}}(t)=-\ddot{f}(-t)=-\phi(f(-t),g(-t))= -\phi(-\tilde{f}(t),\tilde{g}(t))=\phi(\tilde{f}(t),\tilde{g}(t))$$

and

$$\ddot{\tilde{g}}(t)=\ddot{g}(-t)=\xi(f(-t),g(-t))= \xi(-\tilde{f}(t),\tilde{g}(t))=\xi(\tilde{f}(t),\tilde{g}(t))$$

By the uniqueness of the solutions of ordinary differential equations we get that $f(t)=\tilde{f}(t)=-f(-t)$ and $g(t)=\tilde{g}(t)=g(-t)$. This finishes the proof.

\end{proof}

\begin{lem}\label{per2}
Let us assume that $f$ and $g$ satisfies the ordinary differential equation 
\begin{eqnarray*}
\ddot{f}&=&\phi(f,g)\\
\ddot{g}&=&\xi(f,g)
\end{eqnarray*}

If $\dot{f}(0)=0$  and $\dot{g}(0)=0$ and $\phi$ and $\xi$ smooth functions, then, $f(t)$ and $g(t)$ are even. 
\end{lem}

\begin{proof} Let us consider the functions $\tilde{f}(t)=f(-t)$ and $\tilde{g}(t)=g(-t)$. A direct computation shows that $\tilde{f}(0)=f(0)$, $\dot{\tilde{f}}(0)=0$, $\tilde{g}(0)=g(0)$ and  $\dot{\tilde{g}}(0)=0$. Moreover we have that 

$$\ddot{\tilde{f}}(t)=\ddot{f}(-t)=\phi(f(-t),g(-t))=\phi(\tilde{f}(t),\tilde{g}(t))$$

and

$$\ddot{\tilde{g}}(t)=\ddot{g}(-t)=\xi(f(-t),g(-t))= \xi(\tilde{f}(t),\tilde{g}(t))$$

By the uniqueness of the solutions of ordinary differential equations we get that $f(t)=\tilde{f}(t)=f(-t)$ and $g(t)=\tilde{g}(t)=g(-t)$. This finishes the proof.

\end{proof}

This is the Poincare-Miranda theorem for two variables

\begin{thm} \label{PMT} Let $U$ be an open set in $R^2$ that contains the rectangle $[a_1,a_2]\times [t_1,t_2]$.  Let us further assume that  $F:U\times I\longrightarrow R$ and $G:U\times I\longrightarrow R$ are continuous functions. If 

(i) $F(a,t_1)>0$ and $F(a,t_2)<0$ for all $a\in[a_1,a_2]$,

 (ii) $G(a_1,t)<0$ and $G(a_2,t)>0$  for all $t\in[t_1,t_2]$,

 then, there exists a point $(a_0,t_0)\in [a_1,a_2]\times [t_1,t_2]$ such that $F(a_0,t_0)=0=G(a_0,t_0)$. 
\end{thm}

\section{The solution of the ODE as a function of the time and the parameters $a$ and $b$.}\label{ws}

The main tool used to show the periodicity of the solution of the three body problem is to understand the functions in the solution of an ODE as functions of $t$, $a$ and $b$. We this in mind we define,

\begin{mydef} \label{tf} We will denote by $F(t,a,b)=f(t)$ and $R(t,a,b)=r(t)$ and $\Theta(t,a,b)=\theta(t)$ the solution of the system (\ref{e1}) with initial conditions 

$$f(0)=0,\quad \dot{f}(0)=b,\quad r(0)=10, \quad \dot{r}(0)=0 \quad \hbox{and} \quad  \theta(0)=0$$ 

We will denote by $\dot{R}=\frac{\partial R}{\partial t}$ and $\dot{F}=\frac{\partial F}{\partial t}$ and in general if $S$ is a function of the variable $(t,a,b)$ then $\dot{S}=\frac{\partial S}{\partial t}$

\end{mydef}

It is well know that the function $F$ and $R$ has continuous partial derivatives \cite{G} and they obey a differential equation. 

\begin{thm} \label{wst}
If we denote by $F_a=\frac{\partial F}{\partial t}$, $F_b=\frac{\partial F}{\partial b}$, $R_a=\frac{\partial R}{\partial a}$ and $R_b=\frac{\partial R}{\partial b}$, then for any fixed values $a$ and $b$, 

\begin{eqnarray*}
Y(t)&=&(F(t,a,b),\dot{F}(t,a,b),R(t,a,b),\dot{R}(t,a,b),F_a(t,a,b),\dot{F_a}(t,a,b),\\
& &R_a(t,a,b),\dot{R_a}(t,a,b),F_b(t,a,b),\dot{F_b}(t,a,b),R_b(t,a,b),\dot{R_b}(t,a,b))
\end{eqnarray*}

satisfies the differential equation $\dot{Y}=P(Y)$ with $P(x_1,\dots,x_{12})=P(x)=(P^1(x),\dots , P^{12}(x))$ where 

\begin{eqnarray*}
P^1(x)&=& x_2\\
P^2(x)&=&-\frac{400 x_1}{\left(4 x_1^2+x_3^2\right){}^{3/2}} \\
P^3(x)&=&x_4 \\
P^4(x)&=& \frac{100 a^2}{x_3^3}-\frac{25}{x_3^2}-\frac{200 x_3}{\left(4 x_1^2+x_3^2\right){}^{3/2}}\\
P^5(x)&=&x_6 \\
P^6(x)&=& \frac{3200 x_5 x_1^2}{\left(4 x_1^2+x_3^2\right){}^{5/2}}+\frac{1200 x_3 x_7 x_1}{\left(4 x_1^2+x_3^2\right){}^{5/2}}-\frac{400 x_3^2 x_5}{\left(4 x_1^2+x_3^2\right){}^{5/2}}\\
P^7(x)&=& x_8\\
P^8(x)&=&-\frac{300 a^2 x_7}{x_3^4}+\frac{200 a}{x_3^3}+\frac{2400 x_1 x_3 x_5}{\left(4 x_1^2+x_3^2\right){}^{5/2}}+\frac{50 x_7}{x_3^3}+\frac{600 x_3^2 x_7}{\left(4 x_1^2+x_3^2\right){}^{5/2}}-\frac{200 x_7}{\left(4 x_1^2+x_3^2\right){}^{3/2}} \\
P^9(x)&=& x_{10}\\
P^{10}(x)&=& \frac{3200 x_9 x_1^2}{\left(4 x_1^2+x_3^2\right){}^{5/2}}+\frac{1200 x_3 x_{11} x_1}{\left(4 x_1^2+x_3^2\right){}^{5/2}}-\frac{400 x_3^2 x_9}{\left(4 x_1^2+x_3^2\right){}^{5/2}}\\
P^{11}(x)&=& x_{12} \\
P^{12}(x)&=& -\frac{300 a^2 x_{11}}{x_3^4}+\frac{2400 x_1 x_3 x_9}{\left(4 x_1^2+x_3^2\right){}^{5/2}}+\frac{50 x_{11}}{x_3^3}+\frac{600 x_3^2 x_{11}}{\left(4 x_1^2+x_3^2\right){}^{5/2}}-\frac{200 x_{11}}{\left(4 x_1^2+x_3^2\right){}^{3/2}}
\end{eqnarray*}

with initial conditions

$$Y(0)=(0,b,10,0,   0,0,0,0,  0, 1,0,0)$$
\end{thm}

\section{Main result}\label{mr}
In this section we explain the arguments used in the proof of the periodicity of the solution and we will also show how to reduced this proof to the proof of 4 lemmas. With the intension of not cutting the flow of main ideas, we will prove  these lemmas in a different section.

\subsection{Special values used in the proof} \label{consts} before we continue explaining the main result we would like to defined some constants.

\begin{eqnarray*}
a_0&=&\frac{43170475352787}{10000000000000}\quad da=\frac{17}{50000000}\quad sa=\frac{1197}{100000000}\\
b_0&=&\frac{1490359743}{1000000000} \quad sb=\frac{1}{50000}\\
t_0&=&\frac{13366894627923}{5000000000000}\quad dt=\frac{1}{2500000}\quad st=\frac{11}{2000000}
\end{eqnarray*}

\subsection{Main Theorem: Periodic solution}

Our main  theorem shows that for some values of $a$, $b$ and $T$ of the form $a=4.3170\dots$, $b=1.4903\dots$ and $T=96.24\dots$, the solution of the three body problem is periodic.    For this periodic solution the orbit of the body with mass $200$ goes up and down on the $z$-axis, and the other two bodies moves around the $z$-axis; the orbit of one of them is shown in Figure \ref{fig1}.

\begin{thm}\label{mt}
There exist a  triple $(\bar{a},\bar{b},\bar{t})$ with $|\bar{a}-a_0|<sa+da$, $|\bar{b}-b_0|<sb$ and $|\bar{t}-t_0|<6 (ds+dt)$, such that  the solution of the three body problem given by 

\begin{eqnarray*}
x(t)&=&(0,0,F(t,\bar{a},\bar{b}))\\
y(t)&=&(R(t,\bar{a},\bar{b}) \cos(\Theta(t,\bar{a},\bar{b})),R(t,\bar{a},\bar{b}) \sin(\Theta(t,\bar{a},\bar{b})),- F(t,\bar{a},\bar{b}))\\
z(t)&=&(-R(t,\bar{a},\bar{b}) \cos(\Theta(t,\bar{a},\bar{b})),-R(t,\bar{a},\bar{b})\sin(\Theta(t,\bar{a},\bar{b})),-F(t,\bar{a},\bar{b}))\, ,
\end{eqnarray*}

is period with period $36 \bar{t}$. Moreover, we have that $\Theta(\bar{a},\bar{b},\bar{t})=\frac{7 \pi}{18}$.

\end{thm}

\subsection{Symmetry lemmas and reduced periodic solutions }

A solution of the three body problem is called {\it reduced} periodic if the functions that provide the distances between the bodies are periodic with the same period $T_0$. In the subfamily of solutions that we are describing in this paper,  reduced periodic means that for some $T_0>0$, $R(t+kT_0,a,b)=R(t,a,b)$ and $f(t+kT_0,a,b)=f(t,a,b)$ for any integer $k$. The following theorem makes easier the task of finding reduced periodic solutions.

\begin{lem}\label{symlem}
If for some $t_1>0$ we have that $\dot{R}(t_1,a,b)=0=\dot{F}(t_1,a,b)$, then the functions $f(t)=F(t,a,b)$ and $r(t)=R(t,a,b)$ are periodic with period $4t_1$.
\end{lem}

\begin{proof}
By Lemma (\ref{per2}) we have that  $F(t_1-t,a,b)=F(t_1+t,a,b)$ and $R(t_1-t,a,b)=R(t_1+ta,b)$. It follows that  $-\dot{F}(t_1-t,a,b)=\dot{F}(t_1+t,a,b)$ and $-\dot{R}(t_1-t,a,b)=\dot{R}(t_1+t,a,b)$, therefore $\dot{R}(2t_1,a,b)=\dot{R}(0,a,b)=0$ and $F(2t_1,a,b)=F(0,a,b)=0$. By Lemma \ref{per} we have that $f(-t)=-f(t)$ and $r(-t)=r(t)$.  Using Lemma \ref{per} one more time we obtain that $f(t+ 2 t_1)=-f(2t_1-t)$ and $r(2t_1+t)=r(2t_1-t)$. It follows that
$f(4 t_1)=f(0)=0$, $\dot{f}(4 t_1)=\dot{f}(0)=b$, $r(4 t_1)=r(0)=10$ and $\dot{r}(4t_1)=-\dot{r}(0)=0$. Since the values of the solutions $f$ and $r$ at $t=4t_1$ agrees with those at $t=0$, then the lemma follows.
\end{proof}

By the previous theorem, a point $(t_1,a_1,b_1)$ satisfying $\dot{R}(t_1,a_1,b_1)=0=\dot{F}(t_1,a_1,b_1)$ defines a reduced periodic solution. If additionally $\Theta(t_1,a_1,b_1)=\frac{2 \pi p}{q}$ with $p$ and $q>0$ integers, then the solution is not only reduced periodic but periodic with period $4 q t_1$. Moreover, by the implicit function theorem if the cross product between the gradients of the functions $\dot{R}$ and $\dot{F}$ at $(t_1,a_1,b_1)$ does not vanish, then there is a curve of points that solve the equation $\dot{R}(t,a,b)=0=\dot{F}(t,a,b)$ and then, we obtain a family of reduced periodic solutions. In the case that the function $\Theta$ is not constant along this curve of points in the space that represent reduced periodic solutions, then we obtain infinitely many periodic solution due to the fact that on any open interval there are infinitely many numbers of the form $\frac{2 \pi p}{q}$ with $p$ and $q$ integers. The Implicit function theorem also tell us that there exists an small open set around $(t_1,a_1,b_1)$ such that all the solutions of the equations $\dot{R}(t,a,b)=0=\dot{F}(t,a,b)$ in this small open set must be  part of this curve... but how small is small?  Notice that the following two difficulties need to be taken care of: (i) The fact that $\Theta(t_1,a_1,b_1)$ is near $\frac{7 \pi}{18}$ and $\Theta$ is not constant along the curve, does not imply that $\Theta$ eventually reaches the value $\frac{7 \pi}{18}$ on this  curve. (ii) The fact that  $\Theta(t_1,a_1,b_1)<\frac{7 \pi}{18}$ and $\Theta(t_2,a_2,b_2)>\frac{7 \pi}{18}$ for two nearby points $(t_1,a_1,b_1)$ and $(t_2,a_2,b_2)$ that satisfy the equations $\dot{R}(t,a,b)=0=\dot{F}(t,a,b)$, does not guarantee that $\Theta$ eventually reaches the value $\frac{7 \pi}{18}$ due to the fact the point $(t_1,a_1,b_1)$ may no be in the same connected component of the curve of solutions that contains  $(t_2,a_2,b_2)$. Theorem \ref{iftt} helps to solve these two difficulties.

\begin{figure}[hbtp]\label{fig3}
\begin{center}\includegraphics[width=.6\textwidth]{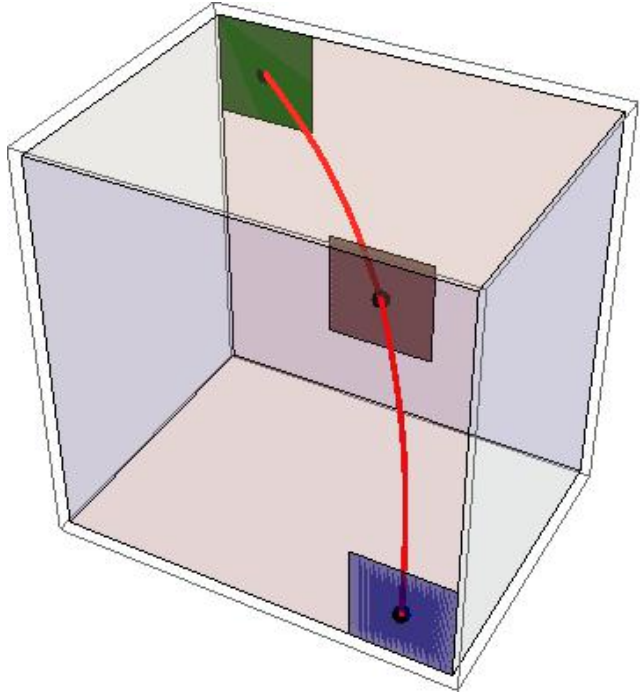}
\end{center}
\caption{Points $(t,a,b)$ in the curve represent reduced periodic solutions of the three body problem. We proved the existence of a point in this curve in each one of the three small rectangles shown in the picture and we prove that at the point in the curve contained in  the first rectangle the value of the function $\Theta$ is smaller than $\frac{7\pi}{18}$ and at the point in the curve contained in  the last rectangle the value of  the function $\Theta$ is bigger than $\frac{7\pi}{18}$. The proof of the existence of a point in the curve  on the rectangle in the center was needed to implement our theorem regarding  the Implicit Function Theorem, which, in this case, it guarantees  the existence of  a unique point on the curve in every plane $b=c$.}\label{fig3}
\end{figure}

The idea of the proof of  Theorem \ref{mt} is the following. We consider a box in the space, see Figure \ref{fig3}. We use Theorem \ref{iftt} to show that every plane $b=c$ in this box contains only one solution of the equation $\dot{R}(t,a,b)=0=\dot{F}(t,a,b)$. We use the Poincare Miranda Theorem to show that there are reduced periodic solutions on three rectangles contained in planes of the form $b=b_0$, $b=b_0-sb$ and $b=b_0+sb$. We finish the prove by showing that $\Theta$ 
evaluated on the solution of the equations $\dot{R}(t,a,b)=0=\dot{F}(t,a,b)$ contained  on the rectangle with $b=b_0-sb$ is less than $\frac{7 \pi}{18}$
and $\Theta$ evaluated on the solution of the equations $\dot{R}(t,a,b)=0=\dot{F}(t,a,b)$ contained  on the rectangle with $b=b_0+sb$ is greater than $\frac{7 \pi}{18}$.

Now we are ready to state the four lemmas mentioned in the beginning of this section that lead to the proof of the main Theorem \ref{mt}. The first lemma shows that there is a point in the curve on the box in Figure \ref{fig3} in the rectangle in the center. 

\begin{lem} \label{l1} There is a solution of the equations $\dot{R}(t,a,b)=0=\dot{F}(t,a,b)$ with $b=b_0$, $|a-a_0|<da$ and $|t-t_0|<dt$
\end{lem}

The second and third lemmas show that there is a point in the curve on the initial and final rectangles  \ref{fig3}. 

\begin{lem} \label{l2} For some $(t,a,b)$ with $b=b_0-sb$, $|a-(a_0+sa)|<da$ and $|t-(t_0-st)|<dt$, there is a solution of the equations  $\dot{R}(t,a,b)=0=\dot{F}(t,a,b)$ such that $\Theta(t,a,b)<\frac{7 \pi}{18}$. 
\end{lem}

\begin{lem} \label{l3} For some $(t,a,b)$ with $b=b_0+sb$, $|a-(a_0-sa)|<da$ and $|t-(t_0+st)|<dt$, there is a solution of the equations  $\dot{R}(t,a,b)=0=\dot{F}(t,a,b)$ such that $\Theta(t,a,b)>\frac{7 \pi}{18}$. 
\end{lem}

The fourth  lemma guaranties the set for values of $(t,a,b)$ given by 

$$[t_0-6(st+dt),t_0+6(st+dt)]\times [a_0-3(sa+da),a_0+3(sa+da)]\times [b_0-sb,b_0+sb] $$

is small enough to only allow one connected curve on it as the solution for the equation $\dot{R}(t,a,b)=0=\dot{F}(t,a,b)$.
 
\begin{lem} \label{l4} For every $b\in [b_0-sb,b_0+sb]$ there exists a unique solution $(t,a,b)$ of the equations $\dot{R}(t,a,b)=0=\dot{F}(t,a,b)$ with 
$|a-a_0|<3 (sa+da) $ and $|t-t_0|<6(st+dt)$.
\end{lem}

\section{Bounds }

In this section we will define 3 differential equations that will help us find bounds for the functions $F(t,a,b)$, $R(t,a,b)$, $\Theta(t,a,b)$ and their partial derivatives. We will be using the functions $\phi_i$ defined in section \ref{dic}.

 
\subsection{The vector field $W$} Let us consider the vector field $W=(x_2,\phi_1,x_4,\phi_2,\phi_3)^T$,
where the functions $\phi_i$'s are those defined on section \ref{dic}. For any fixed $a$ and $b$, the function

 $$t\to \left( F(t,a,b),\dot{F}(t,a,b),R(t,a,b),\dot{R}(t,a,b),\Theta(t,a,b)\right)^T$$
 
  satisfies the differential equation $Y(t)=W(Y(t))$ with $Y(0)=(0,b,10,0,0)^T $
  
  \begin{prop} Using the functions $\phi_i$ defined in section \ref{dic}, we have that the derivative matrix  $DW$ of the vector field $W$,  and the derivative matrix $DW_1$ of the vector field $W_1=DW \, W$ are given by,
 
  $$ DW= \left(
\begin{array}{ccccc}
 0 & 1 & 0 & 0 & 0 \\
 \phi_4  & 0 & \phi_5 & 0 & 0 \\
 0 & 0 & 0 & 1 & 0 \\
 2 \phi_5 & 0 & \phi_6 & 0 & 0 \\
 0 & 0 & \phi_7 & 0 & 0 \\
\end{array}
\right)\quad\hbox{and}\quad   DW_1= \left(
\begin{array}{ccccc}
 \phi_4 & 0 & \phi_5 & 0 & 0 \\
 \phi_{10}  & \phi_4 & \phi_{12} & \phi_5 & 0 \\
 2 \phi_5 & 0 & \phi_6 & 0 & 0 \\
 2 \phi_{12} & 2\phi_5 & \phi_{14} & \phi_6 & 0 \\
 0 & 0 & \phi_{15} & \phi_{7} & 0 \\
\end{array}
\right)$$
   
  \end{prop}


\subsection{The vector field $G$} Let us consider the vector field $G=(x_2,\phi_1,x_4,\phi_2,x_6,\phi_{16},x_8,\phi_{17},\phi_{18})^T$, where the function $\phi_i$ are those defined on section \ref{dic}. For any fixed $a$ and $b$, the function that sends $t$ to 

 \begin{eqnarray*}
\left( F(t,a,b),\dot{F}(t,a,b),R(t,a,b),\dot{R}(t,a,b),F_a(t,a,b),\dot{F}_a(t,a,b), R_a(t,a,b),\dot{R}_a(t,a,b),\Theta_a(t,a,b)\right)^T
\end{eqnarray*}

 satisfies the differential equation $Y(t)=G(Y(t))$ with $Y(0)=(0,b,10,0,0,0,0,0,0)^T$
  
  \begin{prop} Using the functions $\phi_i$ defined on section \ref{dic}, we have that the derivative matrix  $DG$ of the vector field $G$,  and the derivative matrix $DG_1$ of the vector field $G_1=DG \, G$ are given by,

  $$ DG= \left(
\begin{array}{ccccccccc}
 0 & 1 & 0 & 0 & 0 & 0 & 0 & 0 & 0  \\
 \phi_4  & 0 & \phi_5 & 0 & 0 & 0 & 0 & 0 & 0 \\
 0 & 0 & 0 & 1 & 0 & 0 & 0 & 0 & 0 \\
 2 \phi_5 & 0 & \phi_6 & 0 & 0 & 0 & 0 & 0 & 0\\
 0 & 0 & 0 & 0 & 0 & 1 & 0 & 0 & 0\\
 \phi_{19} & 0 & \phi_{20} & 0 & \phi_4 & 0 & \phi_5 & 0 & 0\\
 0 & 0 & 0 & 0 & 0 & 0 & 0 & 1 & 0\\
 2 \phi_{20} & 0 & \phi_{22} & 0 & 2 \phi_5 & 0 & \phi_6 & 0 & 0\\
 0 & 0 & \phi_{23} & 0 & 0 & 0 & \phi_7 & 0 & 0
\end{array}
\right)$$

and 

  $$ DG_1= \left(
\begin{array}{ccccccccc}
 \phi_4& 0 & \phi_5 & 0 & 0 & 0 & 0 & 0 & 0  \\
 \phi_{10}  & \phi_4 & \phi_{12} & \phi_5  & 0 & 0 & 0 & 0 & 0 \\
 2 \phi_5 & 0 & \phi_6 & 0& 0 & 0 & 0 & 0 & 0 \\
 2 \phi_{12} & 2\phi_5 & \phi_{14} & \phi_6 & 0 & 0 & 0 & 0 & 0\\
 \phi_{19} & 0 & \phi_{20} & 0 & \phi_4 & 0 & \phi_5 & 0 & 0\\
 \phi_{30} & \phi_{19} & \phi_{34} & \phi_{20} & \phi_{10} & \phi_4 & \phi_{12} & \phi_5 & 0\\
 2\phi_{20} & 0 & \phi_{22} & 0 & 2\phi_5 & 0 & \phi_6 & 0 & 0\\
 2\phi_{34} & 2\phi_{20} & \phi_{39} & \phi_{22} & 2\phi_{12} & 2\phi_{5} & \phi_{14} & \phi_6 & 0\\
 0 & 0 & \phi_{40} & \phi_{23} & 0 & 0 & \phi_{15} & \phi_7 & 0
\end{array}
\right)$$

\end{prop}


\subsection{The vector field $U$} Let us consider the vector field 

$$U=(x_2,\phi_1,x_4,\phi_2,x_{10},\phi_{41},x_{12},\phi_{42})^T,$$ 
where the function $\phi_i$ are those defined on section \ref{dic}. For any fixed $a$ and $b$ the function

 \begin{eqnarray*}
  t\to& &\left( F(t,a,b),\dot{F}(t,a,b),R(t,a,b),\dot{R}(t,a,b),F_b(t,a,b),\dot{F}_b(t,a,b),\right.\\
 & & \left. \quad R_b(t,a,b),\dot{R}_b(t,a,b)\right)^T\
 \end{eqnarray*}
 
  satisfies the differential equation $Y(t)=U(Y(t))$ with $Y(0)=(0,b,10,0,0,1,0,0)^T$
  
  \begin{prop} Using the functions $\phi_i$ defined in section \ref{dic}, we have that the derivative matrix  $DU$ of the vector field $U$,  and the derivative matrix $DU_1$ of the vector field $U_1=DU \, U$ are given by,

  $$ DU= \left(
\begin{array}{cccccccc}
 0 & 1 & 0 & 0 & 0 & 0 & 0 & 0   \\
 \phi_4  & 0 & \phi_5 & 0 & 0 & 0 & 0 & 0 \\
 0 & 0 & 0 & 1 & 0 & 0 & 0 & 0  \\
 2 \phi_5 & 0 & \phi_6 & 0 & 0 & 0 & 0 & 0 \\
 0 & 0 & 0 & 0 & 0 & 1 & 0 & 0 \\
 \phi_{43} & 0 & \phi_{44} & 0 & \phi_4 & 0 & \phi_5 & 0 \\
 0 & 0 & 0 & 0 & 0 & 0 & 0 & 1 \\
 2 \phi_{44} & 0 & \phi_{45} & 0 & 2 \phi_5 & 0 & \phi_6 & 0 
 \end{array}
\right)$$

and 

  $$ DU_1= \left(
\begin{array}{cccccccc}
  \phi_4& 0 & \phi_5  & 0 & 0 & 0 & 0 & 0   \\
 \phi_{10}  & \phi_4 & \phi_{12} & \phi_5   & 0 & 0 & 0 & 0 \\
2 \phi_5 & 0 & \phi_6 & 0 & 0 & 0 & 0 & 0  \\
  2 \phi_{12} & 2\phi_5 & \phi_{14} & \phi_6  & 0 & 0 & 0 & 0 \\
 \phi_{43} & 0 & \phi_{44} & 0 & \phi_4 &  0 & \phi_5 & 0 \\
 \phi_{48} & \phi_{43} & \phi_{50} & \phi_{44} & \phi_{10} & \phi_4 & \phi_{12} & \phi_5 \\
  2\phi_{44} & 0 & \phi_{45} & 0 & 2\phi_5 & 0 & \phi_6 & 0 \\
 2\phi_{50} & 2\phi_{44} & \phi_{53} & \phi_{45} & 2\phi_{12} & 2\phi_5 & \phi_{14} & \phi_6 
\end{array}
\right)$$
\end{prop}

\subsection{Reduced Periodic Solutions} In this section we use theorem \ref{error} and the Poincare -Miranda Theorem to prove the existence of three reduced periodic solutions, one on each rectangle in figure \ref{fig3}. 

\begin{rem} Several lemmas in this section will be using the Round Taylor method and therefore it will be using Theorem \ref{error} to estimate the values of the solution of the ODE's. If we take a look a the hypothesis of this lemma we notice that there is a number $\epsilon$ that has to be greater than $M_0h+\tilde{H}$. This $\epsilon$ will be $\frac{1}{1000}$ in all proofs that use the Taylor method in this paper. Some of these lemmas contains the variable $\epsilon$, in each case, it just refers to a small number giving an estimate of the error.
\end{rem}

\begin{lem}\label{bds} Let $\epsilon=\frac{2677451}{100000000}$. For any $a\in[a_0-3 (sa+da),a_0+3 (sa+da)]$, $b\in [b_0-sb,b_0+sb]$ and $t\in [t_0 - 6 (st + dt),t_0 + 6(st+ dt)]$ we have that 

$$\displaystyle{|\Theta_a(t,a,b)-\frac{536760312951}{20000000000000}|<\epsilon}$$ 

and,

$$|F_a(t,a,b)-\frac{3032500537707}{10000000000000}|<\epsilon,\quad |\dot{F}_a(t,a,b)-\frac{11824770099363}{25000000000000}|<\epsilon$$

$$ |R_a(t,a,b)-\frac{68073031375453}{25000000000000}|<\epsilon,\quad |\dot{R}_a(t,a,b)-\frac{164497338366219}{100000000000000}|<\epsilon$$ 
\end{lem}
\begin{proof} 

Let us consider the following intervals, 

$$I_1=\left[-\frac{1}{100},\frac{62}{25}\right]\, I_2=\left[-\frac{1}{100},\frac{3}{2}\right] \, I_3=\left[\frac{189}{20},\frac{1001}{100}\right]\, I_4= \left[-\frac{33}{100},\frac{1}{100}\right] \, I_5=\left[-\frac{1}{100},\frac{31}{100}\right]$$

$$I_6=\left[-\frac{1}{100},\frac{12}{25}\right] \, I_7=\left[-\frac{1}{100},\frac{69}{25}\right] \, I_8= \left[-\frac{1}{100},\frac{42}{25}\right] \, J_9=\left[ -\frac{1}{100}, \frac{3}{25} \right]$$

A direct computation using the bounds in section \ref{dic} shows that if

$$ M_{0g}= \frac{42}{25},\quad K_{0g}=\frac{305541}{125000},\quad K_{1g}=\frac{249309}{100000},\quad  M_{1g}= \frac{778131}{1000000} \quad M_{2g}= \frac{246743}{250000}  $$

$$  M_{3g}=\frac{374443}{500000}, \quad  M_{4g}= \frac{133409}{100000} ,\quad  M_{5g}=  \frac{1345793}{1000000}, \quad  M_{6g}=  \frac{2429239}{1000000} $$

$$  M_{7g}=  \frac{833241}{500000}, \quad M_{8g}=  \frac{3298559}{1000000},  \quad  M_{9}=   \frac{182893}{1000000},  \quad H=10^{-14}, \quad h=\frac{t_0}{30000} \, ,$$

then, for values of $(x_1,\dots,x_8)$ with $x_i\in I_i$ we have that $|G^j|<M_{0g}$ where $G^1,\dots , G^9$ are the entries  of the vector field $G$; $|G_2^j|<M_{jg}$ where $G_2^1,\dots , G_2^9$ are the entries  of the vector field $G_2=DG_1\, G$. Recall that $G_1=DG\, G$.  Moreover we have that

$$|DG|<K_{0g},\quad |DG_1|<K_{1g} $$

The Round Taylor method of order 2 using the vector field $G$ with $a=a_0$,  $k=30000$ and initial conditions $Y(0)=(0,b_0,10,0,0,0,0,0,0)^T$ produces a sequence $\{z_i\}_{i=1}^k$ with $z_{30000}$ equal to

\begin{eqnarray*}
& &\hskip-1cm\left(  \frac{247458249564811}{100000000000000},\frac{13245901}{100000000000000},\frac{189061430242601}{20000000000000},\frac{1795639}{12500000000000},\right. \\
& &\hskip-.8cm  \left. \frac{3032500537707}{10000000000000},\frac{11824770099363}{25000000000000},\frac{68073031375453}{25000000000000},\frac{164497338366219}{100000000000000},\frac{536760312951}{20000000000000}\right)
\end{eqnarray*}

In this case $\tilde{H}=\frac{M_g \frac{h^2}{6}+\frac{H}{h}}{L}\left(\hbox{e}^{Lkh}-1\right)\approx 0.0000019$ where $M_g=\sqrt{M_{1g}^2+\dots +M_{9g}^2}$ and $L=K_{0g}+K_{1g} \frac{h}{2}$. A direct verification shows that, for every $j=1,\dots , 8$ and $i=1,\dots, 30000$,  the $j^{th}$ entry of $z_i$ is within a distance $\epsilon$ of the boundary of the interval $I_j$. Also we have that the last entry of $z_i$ is within a distance $\frac{1}{1000}$ of the boundary of the interval $J_9$. By Theorem \ref{error} we conclude that for all $i=1,\dots, k$, 

 \begin{eqnarray*}
& &\left( F(\frac{i h}{k},a_0,b_0),\dot{F}(\frac{i h}{k},a_0,b_0),R(\frac{i h}{k},a_0,b_0),\dot{R}(\frac{i h}{k},a_0,b_0),F_a(\frac{i h}{k},a_0,b_0),\dot{F}_a(\frac{i h}{k},a_0,b_0),\right.\\
 & & \left. \quad R_a(\frac{i h}{k},a_0,b_0),\dot{R}_a(\frac{i h}{k},a_0,b_0),\Theta_a(\frac{i h}{k},a_0,b_0)\right)^T\
 \end{eqnarray*}

is within a distance $\tilde{H}$ of $z_i$. Notice that for any function $\rho$, under the assumption that $|\dot{\rho}(\tau,a_0,b_0)|<M_\rho$ for all $\tau$ between $t$ and $t_0$,  we have that 

\begin{eqnarray*}
|\rho(t,a,b)-\rho(t_0,a_0,b_0)| &\le & |\rho(t,a,b)-\rho(t,a_0,b_0)|+|\rho(t,a_0,b_0)-\rho(t_0,a_0,b_0)|\\
     &\le &  |\rho(t,a,b)-\rho(t,a_0,b_0)| + M_\rho\, (t-t_0),
\end{eqnarray*}

 We will use the observation above to finish the proof of the lemma. We will bound  $|\rho(t,a,b)-\rho(t,a_0,b_0)|$ using Theorem \ref{error2}
with the vector fields $G$ and $G_0$, where $G_0$ is the vector field $G$ with $a$ replaced by $a_0$. A direct computation shows that

$$\delta G=G-G_0=(0,0,0,\phi_{54},0,0,0,\phi_{56},\phi_{57})$$

Using the information on section \ref{dic} we obtain that $|\delta_G|<\frac{23}{500000}$. Therefore, using Theorem \ref{error2} we conclude that the values of the solution of the differential equation using $G$ (with a general $a$ and $b$) compare with those of the solution of the differential equation using $G_0$ differ by less than

$$|b-b_0| \hbox{e}^{K_{0g}(t_0+6(st+dt))}+\frac{23}{500000 K_{0g}}\left( \hbox{e}^{K_{0g}(t_0+6(st+dt))}-1 \right)< \frac{267131}{10000000} $$

Therefore we have that for any $a\in[a_0-3 (sa+da),a_0+3 (sa+da)]$, $b\in [b_0-sb,b_0+sb]$ and $t\in [t_0 - 6 (st + dt),t_0 + 6(st+ dt)]$

$$|F_a(t,a,b)- \frac{3032500537707}{10000000000000}|<\frac{267131}{10000000}+6(st+dt) M_{0g}+\tilde{H}\le\frac{2677451}{100000000}$$

We have similar computations for the functions $\Theta_a$, $\dot{F}_a$, $R_a$ and $\dot{R}_a$. This finishes the proof.
\end{proof}


\begin{cor} \label{dFapos}
For any $a\in[a_0-3 (sa+da),a_0+3 (sa+da)]$, $b\in [b_0-sb,b_0+sb]$ and $t\in [t_0 - 6 (st + dt),t_0 + 6(st+ dt)]$ , $\dot{F}_a(t,a,b)>0$,
$\dot{R}_a(t,a,b)>0$, and $|\Theta_a(t,a,b)|<  \frac{27}{1000}$
\end{cor}


\begin{lem}\label{bdsb} Let $\epsilon= \frac{2568201}{100000000}$. For any $a\in[a_0-3 (sa+da),a_0+3 (sa+da)]$, $b\in [b_0-sb,b_0+sb]$ and $t\in [t_0 - 6 (st + dt),t_0 + 6(st+ dt)]$ we have that

$$|\dot{R}_b(t,a,b)-\frac{88229751956717}{100000000000000}|<\epsilon,\quad |\dot{F}_b(t,a,b)-\frac{50798112898451}{100000000000000}|<\epsilon$$

\end{lem}


\begin{proof} 

Let us consider the following intervals, 

$$I_1=\left[-\frac{1}{100},\frac{62}{25}\right]\, I_2=\left[-\frac{1}{100},\frac{3}{2}\right] \, I_3=\left[\frac{189}{20},\frac{1001}{100}\right]\, I_4= \left[-\frac{33}{100},\frac{1}{100}\right] $$

$$I_9=\left[-\frac{1}{100},\frac{101}{50}\right]\, I_{10}=\left[\frac{1}{2},\frac{101}{100}\right] \,  I_{11}=\left[-\frac{1}{100},\frac{81}{100}\right] \, I_{12}=\left[-\frac{1}{100},\frac{89}{100}\right]$$

A direct computation using section \ref{dic} shows that if

$$ M_{0u}= \frac{3}{2},\quad K_{0u}=\frac{1226931}{500000},\quad K_{1u}=\frac{2557349}{1000000},\quad  M_{1u}= \frac{778131}{1000000}\quad M_{2u}= \frac{246743}{250000} $$

$$  M_{3u}=\frac{374443}{500000}, \quad  M_{4u}=\frac{133409}{100000} ,\quad  M_{5u}=  \frac{765259}{500000}, \quad  M_{6u}= \frac{533571}{200000}$$

$$M_{7u}= \frac{1790753}{1000000}\quad  M_{8u}= \frac{711267}{200000}, \quad H=10^{-14}, \quad h=\frac{t_0}{30000}\, , $$

then, for values of $(x_1,x_2,x_3,x_4,x_9,x_{10},x_{11},x_{12})$ with $x_i\in I_i$ we have that $|U^j|<M_{0g}$ where $U^1,\dots , U^8$ are the entries  of the vector field $U$; $|U_2^j|<M_{ju}$ where $U_2^1,\dots , U_2^8$ are the entries  of the vector field $U_2=DU_1\, U$. Recall that $U_1=DU\, U$.  Moreover we have that


$$|DU|<K_{0u},\quad |DU_1|<K_{1u} $$

The Round Taylor method of order 2 using the vector field $U$ with $a=a_0$,  $k=30000$ and initial conditions $Y(0)=(0,b_0,10,0,0,1,0,0)^T$ produces a sequence $\{z_i\}_{i=1}^k$ with $z_{30000}$ equal to

\begin{eqnarray*}
& &\hskip-1cm\left(  \frac{247458249564811}{100000000000000},\frac{13245901}{100000000000000},\frac{189061430242601}{20000000000000},\frac{1795639}{12500000000000},\right. \\
& &\hskip-.8cm \left.  \frac{25138479462137}{12500000000000},\frac{50798112898451}{100000000000000},\frac{20014508374143}{25000000000000},\frac{88229751956717}{100000000000000}\right)
\end{eqnarray*}

In this case $\tilde{H}=\frac{M_u \frac{h^2}{6}+\frac{H}{h}}{L}\left(\hbox{e}^{Lkh}-1\right)< \frac{209}{100000000}$ where $M_u=\sqrt{M_{1u}^2+\dots +M_{8u}^2}$ and $L=K_{0u}+K_{1u} \frac{h}{2}$. A direct verification shows that, for every $j=1,\dots , 8$ and $i=1,\dots, 30000$, 
the $j^{th}$ entry of $z_i$ is within a distance $\frac{1}{1000}$ of the boundary of the interval $I_j$ when $j=1,2,3,4$ and the $j^{th}$ entry of $z_i$ is within a distance $\frac{1}{1000}$ of the boundary of the interval $I_{j+4}$ when $j=5,6,7,8$. By Theorem \ref{error} we conclude that for all $i=1,\dots, k$, 

 \begin{eqnarray*}
& &\left( F(\frac{i h}{k},a_0,b_0),\dot{F}(\frac{i h}{k},a_0,b_0),R(\frac{i h}{k},a_0,b_0),\dot{R}(\frac{i h}{k},a_0,b_0),F_b(\frac{i h}{k},a_0,b_0),\dot{F}_b(\frac{i h}{k},a_0,b_0),\right.\\
 & & \left. \quad R_b(\frac{i h}{k},a_0,b_0),\dot{R}_b(\frac{i h}{k},a_0,b_0)\right)^T\
 \end{eqnarray*}

is within a distance $\tilde{H}$ of $z_i$. Notice that for any function $\rho$, under the assumption that $|\dot{\rho}(\tau,a_0,b_0)|<M_\rho$ for all $\tau$ between $t$ and $t_0$,  we have that 

\begin{eqnarray*}
|\rho(t,a,b)-\rho(t_0,a_0,b_0)| &\le & |\rho(t,a,b)-\rho(t,a_0,b_0)|+|\rho(t,a_0,b_0)-\rho(t_0,a_0,b_0)|\\
     &\le &  |\rho(t,a,b)-\rho(t,a_0,b_0)| + M_\rho\, (t-t_0),
\end{eqnarray*}


 We will use the observation above to finish the proof of the lemma. We will bound  $|\rho(t,a,b)-\rho(t,a_0,b_0)|$ using Theorem \ref{error2}
with the vector fields $U$ and $U_0$, where $U_0$ is the vector field $U$ with $a$ replaced by $a_0$. A direct computation shows that
$$\delta U=U-U_0=(0,0,0,\phi_{54},0,0,0,\phi_{58})$$

Using the information on section \ref{dic} we obtain that $|\delta_U|<\frac{1}{25000}$. Therefore, using Theorem \ref{error2} we conclude that the values of the the solution of the differential equation using $U$ (with a general $a$ and $b$) compare with those of the solution of the differential equation using $U_0$ differ by less than

$$|b-b_0| \hbox{e}^{K_{0u}(t_0+6(st+dt))}+\frac{1}{25000 K_{0u}}\left( \hbox{e}^{K_{0u}(t_0+6(st+dt))}-1 \right)< \frac{1281341}{50000000}$$

Therefore we have that for any $a\in[a_0-3 (sa+da),a_0+3 (sa+da)]$, $b\in [b_0-sb,b_0+sb]$ and $t\in [t_0 - 6 (st + dt),t_0 + 6(st+ dt)]$

$$|\dot{F}_b(t,a,b)- \frac{50798112898451}{100000000000000}|<\frac{1281341}{50000000}+6(st+dt) M_{0u}+\tilde{H}\le  \frac{2568201}{100000000} $$

We have a  similar computation for the function $\dot{R}_b$. This finishes the proof.
\end{proof}





\begin{lem}\label{bdsw} Let $\epsilon=\frac{134567}{40000000}$. For any $a\in[a_0-3 (sa+da),a_0+3 (sa+da)]$, $b\in [b_0-sb,b_0+sb]$ and $t\in [t_0 - 6 (st + dt),t_0 + 6(st+ dt)]$ we have that 

$$|F(t,a,b)-\frac{247458249564811}{100000000000000}|<\epsilon,\quad |R(t,a,b)-\frac{189061430242601}{20000000000000}|<\epsilon$$
 
\end{lem}
\begin{proof} 

Let us consider the following intervals $J_5=[-\frac{1}{100},\frac{7}{5}]$ and $I_1$, $I_2$, $I_3$, $I_4$ defined on the proof of Lemma \ref{bds}. 

A direct computation using the bounds on section \ref{dic} shows that if

$$ M_{0w}= \frac{3}{2},\quad K_{0w}=\frac{25282}{15625},\quad K_{1w}=\frac{260901}{200000},\quad  M_{1w}= \frac{19453263}{25000000}
 \quad M_{3w}=\frac{14977713}{20000000} $$

$$ M_{2w}= \frac{493485626283}{500000000000} \quad  M_{4w}= \frac{1334089805457}{1000000000000} ,\quad  M_{5w}= \frac{5396988231}{125000000000} \, ,$$


then, for values of $(x_1,\dots,x_{4})$ with $x_i\in I_i$ we have that $|W^j|<M_{0w}$ where $W^1,\dots , W^5$ are the entries  of the vector field $W$; $|W_2^j|<M_{jw}$ where $W_2^1,\dots , W_2^5$ are the entries  of the vector field $W_2=DW_1\, W$. Recall that $W_1=DW\, W$.  Moreover we have that

$$|DW|<K_{0w},\quad |DW_1|<K_{1w} $$

The Round Taylor method of order 2 using the vector field $W$ with $ H=10^{-14}$, $h=\frac{t_0}{30000}$, $a=a_0$,  $k=30000$ and initial conditions $Y(0)=(0,b_0,10,0,0)^T$ produces a sequence $\{z_i\}_{i=1}^k$ with $z_{30000}$ equal to

$$ \left(\frac{247458249564811}{100000000000000},\frac{13245901}{100000000000000},\frac{189061430242601}{20000000000000},\frac{1795639}{12500000000000},\frac{12217304404331}{10000000000000}\right)$$

In this case $\tilde{H}=\frac{M_w \frac{h^2}{6}+\frac{H}{h}}{L}\left(\hbox{e}^{Lkh}-1\right)< \frac{127}{1000000000}$ where $M_w=\sqrt{M_{1w}^2+\dots +M_{5w}^2}$ and $L=K_{0w}+K_{1w} \frac{h}{2}$. A direct verification shows that, for every $j=1,\dots , 4$ and $i=1,\dots, 30000$,  the $j^{th}$ entry of $z_i$ is within a distance $\frac{1}{1000}$ of the boundary of the interval $I_j$. Also we have that the last entry of $z_i$ is within a distance $\frac{1}{1000}$ of the boundary of the interval $J_5$. By Theorem \ref{error} we conclude that for all $i=1,\dots, k$, 

 \begin{eqnarray*}
\left( F(\frac{i h}{k},a_0,b_0),\dot{F}(\frac{i h}{k},a_0,b_0),R(\frac{i h}{k},a_0,b_0),\dot{R}(\frac{i h}{k},a_0,b_0),\Theta_a(\frac{i h}{k},a_0,b_0)\right)^T
 \end{eqnarray*}

is within a distance $\tilde{H}$ of $z_i$. Notice that for any function $\rho$, under the assumption that $|\dot{\rho}(\tau,a_0,b_0)|<M_\rho$ for all $\tau$ between $t$ and $t_0$,  we have that 

\begin{eqnarray*}
|\rho(t,a,b)-\rho(t_0,a_0,b_0)| &\le & |\rho(t,a,b)-\rho(t,a_0,b_0)|+|\rho(t,a_0,b_0)-\rho(t_0,a_0,b_0)|\\
     & \le &  |\rho(t,a,b)-\rho(t,a_0,b_0)| + M_\rho\, (t-t_0),
\end{eqnarray*}

 We will use the observation above to finish the proof of the lemma. We will bound  $|\rho(t,a,b)-\rho(t,a_0,b_0)|$ using Theorem \ref{error2}
with the vector fields $W$ and $W_0$, where $W_0$ is the vector field $W$ with $a$ replaced by $a_0$. A direct computation shows that

$$\delta W=W-W_0=(0,0,0,\phi_{54},\phi_{55})$$

Using the information on section \ref{dic} we obtain that $|\delta_W|<\frac{39}{1000000}$. Therefore, using Theorem \ref{error2} we conclude that the values of the the solution of the differential equation using $W$ (with a general $a$ and $b$) compare with those of the solution of the differential equation using $W_0$ differ by less than

$$|b-b_0| \hbox{e}^{K_{0w}(t_0+6(st+dt))}+\frac{39}{1000000 K_{0w}}\left( \hbox{e}^{K_{0w}(t_0+6(st+dt))}-1 \right)< 
\frac{827737}{250000000}$$

Therefore we have that for any $a\in[a_0-3 (sa+da),a_0+3 (sa+da)]$, $b\in [b_0-sb,b_0+sb]$ and $t\in [t_0 - 6 (st + dt),t_0 + 6(st+ dt)]$

$$|F(t,a,b)- \frac{247458249564811}{100000000000000} |<\frac{827737}{250000000}+6(st+dt) M_{0w}+\tilde{H}\le
\frac{134567}{40000000}$$

We have a similar computation for the function $R$. This finishes the proof. \end{proof}


\begin{mydef} We will denote by $Z_W(t,b,a,k)\in R^5$ the last vector in the sequence $\{z_i\}_{i=1}^k$ produced using the Round Taylor Method of order 2 using the vector field $W$ with $H=10^{-14}$, $h=\frac{t}{k}$ and  initial conditions $(0,b,10,0,0)^T$. We denote by $H_W(t,k)=\frac{M_w \frac{h^2}{6}+\frac{H}{h}}{L}\left(\hbox{e}^{Lkh}-1\right)$ where $M_w=\sqrt{M_{1w}^2+\dots +M_{5w}^2}$ and $L=K_{0w}+K_{1w} \frac{h}{2}$. If we need to use  $H=10^{-q}$ instead of $H=10^{-14}$, then  we will use the notation $Z_W(t,b,a,k,q)\in R^5$ and $H_W(t,k,q)$.

\end{mydef}

\begin{rem} For all the following Lemmas that use the Round Taylor Method, it can be directly verified that all the conditions of Theorem \ref{error}
are satisfied; therefore $Z_W(t,b,a,k)$ is within a distance $H_W(t,k)$ of 

$$(F(t,a,b),\dot{F}(t,a,b),R(t,a,b),\dot{R}(t,a,b),\Theta(t,a,b))^T.$$ For the sake of keeping the flow of the proof, we will left out all the details in the proofs of these lemmas. 
\end{rem}

\begin{cor}\label{ddot}  For any $a\in[a_0-3 (sa+da),a_0+3 (sa+da)]$, $b\in [b_0-sb,b_0+sb]$ and $t\in [t_0 - 6 (st + dt),t_0 + 6(st+ dt)]$ we have that $\, \displaystyle{\frac{120689}{250000} <\dot{\Theta}(t,a,b)<\frac{483453}{1000000} }\, $,

$$-\frac{163169}{200000}<\ddot{F}(t,a,b)<  -\frac{813693}{1000000} \quad \hbox{and}
\quad\frac{4593}{12500} <\ddot{R}(t,a,b)<  \frac{18653}{50000}$$

\end{cor}

\begin{proof} As in Lemma \ref{bdsw}, let $\epsilon=\frac{134567}{40000000}$. Notice that the values of  $\ddot{F}(t,a,b)$ are bounded by the minimum and maximum of the function $-\frac{400 x_1}{\left(4 x_1^2+x_3^2\right){}^{3/2}}$ when $|x_1 -\frac{247458249564811}{100000000000000}|<\epsilon$ and 
 $|x_3-\frac{189061430242601}{20000000000000}|<\epsilon$. The bound for $\ddot{F}$ follows as a direct application of the Lagrange Multiplier method. Since the values of  $\ddot{R}(t,a,b)$ are bounded by the minimum and maximum of the function $\frac{100 a^2}{x_3^3}-\frac{25}{x_3^2}-\frac{200 x_3}{\left(4 x_1^2+x_3^2\right){}^{3/2}}$, we can use the same argument for $\ddot{R}(t,a,b)$. Finally since the values of  $\dot{\Theta}(t,a,b)$ are bounded by the minimum and maximum of the function $\frac{10 a}{x_3^2}$, we can also use the same argument for $\dot{\Theta}(t,a,b)$.
\end{proof}

\begin{lem}\label{pmlforb0psb} $\dot{F}(t_0-dt,a,b_0)>0$ for all  $a\in[a_0-da,a_0+da]$, $\dot{F}(t_0+dt,a,b_0)<0$ for all  $a\in[a_0-da,a_0+da]$, $\dot{R}(t,a_0-da,b_0)<0$ for all $t\in[t_0-dt,t_0+dt]$ and $\dot{R}(t,a_0+da,b_0)>0$ for all $t\in[t_0-dt,t_0+dt]$.
\end{lem}
\begin{proof}
By Corollary \ref{dFapos} we have that $\dot{F}_a(t_0-dt,a,b_0)>0$ for all $a\in[a_0-da,a_0+da]$. Therefore, in order to show that $\dot{F}(t_0-dt,a,b_0)>0$ for all  $a\in[a_0-da,a_0+da]$, it is enough to show that $\dot{F}(t_0-dt,a_0-da,b_0)>0$. A direct computation shows that
 $Z_W(t_0-dt,a_0-da,b_0,35000)$ equals to 


$$\left(\frac{7733069351623}{3125000000000},\frac{25787091}{20000000000000},\frac{189061375789453}{20000000000000},-\frac{44841643}{20000000000000},\frac{30543236182739}{25000000000000}\right)$$

and

$$\tilde{H}_W(t_0-dt,35000)<\frac{94851}{1000000000000}$$

Therefore,

$$\dot{F}(t_0-dt,a_0-da,b_0)> \frac{25787091}{20000000000000}- \frac{94851}{1000000000000}> 0$$

As pointed out above, we conclude that  $\dot{F}(t_0-dt,a,b_0)>0$ for all $a\in[a_0-da,a_0+da]$. In the same way we have that $\dot{F}(t_0+dt,a,b_0)<0$ for all  $a\in[a_0-da,a_0+da]$ because a direct computation shows that $Z_W(t_0+dt,a_0+da,b_0,120000)$ equals to 

$$ \left(\frac{1546614123963}{625000000000},-\frac{198811}{6250000000000},\frac{945307243792047}{100000000000000},\frac{16995193}{20000000000000},\frac{61086532113189}{50000000000000}\right) $$

and 

$$\tilde{H}_W(t_0+dt,120000)<\frac{5651}{200000000000}$$

Therefore,

$$\dot{F}(t_0+dt,a_0+da,b_0)<-\frac{198811}{6250000000000}+\frac{5651}{200000000000}<0$$

Let us show that $\dot{R}(t,a_0-da,b_0)<0$ for all  $t\in[t_0-dt,t_0+dt]$. A direct computation shows that $Z_W(t_0+dt,a_0-da,b_0,35000)$ equals to 

$$\left(\frac{123729109625969}{50000000000000},-\frac{196972647}{100000000000000},\frac{945306878946787}{100000000000000},-\frac{19027313}{25000000000000},\frac{122173137972697}{100000000000000}\right) $$

and 

$$\tilde{H}_W(t_0+dt,35000)<\frac{94851}{1000000000000}$$

Therefore,

$$\dot{R}(t_0+dt,a_0-da,b_0)<-\frac{19027313}{25000000000000}+\frac{94851}{1000000000000}<0$$

 Since, $\ddot{R}(t,a_0-da,b_0)>0$ (see Corollary \ref{ddot}), we conclude that $\dot{R}(t,a_0-da,b_0)<0$ for all  $t\in[t_0-dt,t_0+dt]$.
Let us show that $\dot{R}(t,a_0+da,b_0)>0$ for all  $t\in[t_0-dt,t_0+dt]$. A direct computation shows that $Z_W(t_0-dt,a_0+da,b_0,35000)$ equals to 

$$ \left(\frac{24745827990179}{10000000000000},\frac{55883369}{25000000000000},\frac{189061484706171}{20000000000000},\frac{52393253}{50000000000000},\frac{61086475049353}{50000000000000}\right) $$

and 

$$\tilde{H}_W(t_0-dt,35000)<\frac{94851}{1000000000000}  $$

Therefore,

$$\dot{R}(t_0-dt,a_0+da,b_0)<\frac{52393253}{50000000000000}-\frac{94851}{1000000000000}>0$$

 Since, $\ddot{R}(t,a_0+da,b_0)>0$ (see Corollary \ref{ddot}), we conclude that $\dot{R}(t,a_0+da,b_0)>0$ for all  $t\in[t_0-dt,t_0+dt]$. This finishes het proof of the Lemma.

\end{proof}

\begin{rem} \label{pl1}
As a Corollary, using the Poincare-Miranda theorem, we obtain a proof of  Lemma \ref{l1}.
\end{rem}


\begin{lem} \label{pmlforb0msb} $\dot{F}(t_0-st-dt,a,b_0-sb)>0$ for all  $a\in[a_0+sa-da,a_0+sa+da]$, $\dot{F}(t_0-st+dt,a,b_0-sb)<0$ for all  $a\in[a_0+sa-da,a_0+sa+da]$, $\dot{R}(t,a_0+sa-da,b_0-sb)<0$ for all $t\in[t_0-st-dt,t_0-st+dt]$ and $\dot{R}(t,a_0+sa+da,b_0-sb)>0$ for all $t\in[t_0-st-dt,t_0-st+dt]$. Moreover,  for all all $t\in[t_0-st-dt,t_0-st+dt]$ and all $a\in[a_0+sa-da,a_0+sa+da]$, $\Theta(t,a,b_0-sb)<\frac{7 \pi}{18}$.
\end{lem}

\begin{proof}
The proof is similar to that of Lemma \ref{pmlforb0psb}. In this case $\dot{F}(t_0-st-dt,a,b_0-sb)>0$ for all  $a\in[a_0+sa-da,a_0+sa+da]$ because $Z_W(t_0-st-dt,a_0+sa-da,b_0-sb,35000)$ equals to

$$\left(\frac{247454580109467}{100000000000000},\frac{14064311}{50000000000000},\frac{945308716843341}{100000000000000},-\frac{27791851}{50000000000000},\frac{7635805749249}{6250000000000}\right)$$

and 

$$\tilde{H}_W(t_0-st-dt,35000)<\frac{94849}{1000000000000}  $$

Therefore,

$$\dot{F}(t_0-st-dt,a_0+sa-da,b_0-sb)>\frac{14064311}{50000000000000}- \frac{94849}{1000000000000}>0.$$

$\dot{F}(t_0-st+dt,a,b_0-sb)<0$ for all  $a\in[a_0+sa-da,a_0+sa+da]$, because $Z_W(t_0-st+dt,a_0+sa+da,b_0-sb,120000)$ equals to

$$  \left(\frac{49490920135587}{20000000000000},-\frac{4841111}{100000000000000},\frac{29540903186787}{3125000000000},\frac{21447973}{25000000000000},\frac{122172932415207}{100000000000000}\right) $$

and 

$$\tilde{H}_W(t_0-st+dt,120000)< \frac{5651}{200000000000}$$

Therefore,

$$\dot{F}(t_0-st+dt,a_0+sa+da,b_0-sb)< -\frac{4841111}{100000000000000} + \frac{5651}{200000000000} < 0.$$

$\dot{R}(t,a_0+sa+da,b_0-sb)>0$ for all $t\in[t_0-st-dt,t_0-st+dt]$, because $Z_W(t_0-st-dt,a_0+sa+da,b_0-sb,35000)$ equals to

$$ \left(\frac{123727300365019}{50000000000000},\frac{12058321}{20000000000000},\frac{189061780400321}{20000000000000},\frac{11254969}{20000000000000},\frac{61086446906439}{50000000000000}\right)$$

and 

$$\tilde{H}_W(t_0-st-dt,35000)< \frac{94849}{1000000000000}$$

Therefore,

$$\dot{R}(t_0-st-dt,a_0+sa+da,b_0-sb)> \frac{11254969}{20000000000000} -  \frac{94849}{1000000000000} > 0$$

$\dot{R}(t,a_0+sa-da,b_0-sb)<0$ for all $t\in[t_0-st-dt,t_0-st+dt]$, because $Z_W(t_0-st+dt,a_0+sa-da,b_0-sb,35000)$ equals to

$$\left(\frac{4949091602187}{2000000000000},-\frac{18526257}{50000000000000},\frac{945308716843239}{100000000000000},-\frac{25964629}{100000000000000},\frac{122172930636361}{100000000000000}\right) $$

and 

$$\tilde{H}_W(t_0-st-dt,35000)< \frac{1897}{20000000000}$$

Therefore,

$$\dot{R}(t_0-st+dt,a_0+sa-da,b_0-sb)< -\frac{25964629}{100000000000000}+ \frac{1897}{20000000000}  < 0.$$

In order to show that $\Theta(t,a,b_0-sb)<\frac{7\pi}{18}$, we first noticed that $Z_W(t_0-st,a_0+sa,b_0-sb,35000)$ equals to

$$\left(\frac{247454590419723}{100000000000000},\frac{1161959}{10000000000000},\frac{945308809422539}{100000000000000},\frac{473597}{3125000000000},\frac{122172912224601}{100000000000000}\right) $$

and 

$$\tilde{H}_W(t_0-st,35000)<\frac{94849}{1000000000000}$$ . 

Therefore, 

\begin{eqnarray*}\Theta(t_0-st,a_0+sa,b_0-sb)-\frac{7\pi}{18}&<&\frac{122172912224601}{100000000000000}+\frac{94849}{1000000000000} - \frac{7\pi}{18}\\
& =&\frac{122172921709501}{100000000000000}-\frac{7 \pi }{18}
\end{eqnarray*}

Using the mean value theorem, and the bounds $|\Theta_a|<\frac{483453}{1000000}$ and $|\dot{\Theta}|<\frac{27}{1000}$ found in Corollaries \ref{ddot} and \ref{dFapos}, we obtain that  
for any $t\in [t_0-st-dt, t_0-st+dt]$ and $a\in[a_0+sa-da,a_0+sa+da]$

\begin{eqnarray*}
\Theta(t,a,b_0-sb)-\frac{7\pi}{18}& =& \Theta(t,a,b_0-sb)- \Theta(t,a_0-sa,b_0-sb)+\\
& &\Theta(t,a_0+sa,b_0-sb)- \Theta(t_0,a_0+sa,b_0-sb)+\\
& & \Theta(t_0,a_0+sa,b_0-sb)-\frac{7\pi}{18}\\
& <&\frac{483453}{1000000}\, da+ \frac{27}{1000}\, dt + \frac{122172921709501}{100000000000000}-\frac{7 \pi }{18}\\
&<&0
\end{eqnarray*}

\end{proof}

\begin{rem} \label{pl2}
As a Corollary, using the Poincare-Miranda theorem, we obtain a proof of  Lemma \ref{l2}
\end{rem}


\begin{lem} $\dot{F}(t_0+st-dt,a,b_0+sb)>0$ for all  $a\in[a_0-sa-da,a_0-sa+da]$, $\dot{F}(t_0+st+dt,a,b_0+sb)<0$ for all  $a\in[a_0-sa-da,a_0-sa+da]$, $\dot{R}(t,a_0-sa-da,b_0+sb)<0$ for all $t\in[t_0+st-dt,t_0+st+dt]$ and $\dot{R}(t,a_0-sa+da,b_0+sb)>0$ for all $t\in[t_0+st-dt,t_0+st+dt]$. Moreover,  for all all $t\in[t_0+st-dt,t_0+st+dt]$ and all $a\in[a_0-sa-da,a_0-sa+da]$, $\Theta(t,a,b_0+sb)>\frac{7 \pi}{18}$.
\end{lem}

\begin{proof}
The proof is similar to that of Lemma \ref{pmlforb0msb}. In this case $\dot{F}(t_0+st-dt,a,b_0+sb)>0$ for all  $a\in[a_0-sa-da,a_0-sa+da]$ because $Z_W(t_0+st-dt,a_0-sa-da,b_0+sb,35000)$ equals to

$$\left(\frac{247461898442221}{100000000000000},\frac{1573331}{5000000000000},\frac{189061080098229}{20000000000000},-\frac{57188327}{100000000000000},\frac{15271644451087}{12500000000000}\right)$$

and 

$$\tilde{H}_W(t_0+st-dt,35000)<\frac{23713}{250000000000}$$

Therefore,

$$\dot{F}(t_0-st-dt,a_0+sa-da,b_0-sb)>\frac{1573331}{5000000000000}-\frac{23713}{250000000000}>0.$$


$\dot{F}(t_0+st+dt,a,b_0+b)<0$ for all  $a\in[a_0+sa-da,a_0+sa+da]$, because $Z_W(t_0+st+dt,a_0-sa+da,b_0+sb,120000,15)$ equals to

\begin{eqnarray*}
( & &\frac{1237309595658901}{500000000000000},-\frac{904301}{62500000000000},\frac{9453055857540481}{1000000000000000}, \\
& &\frac{842419313}{1000000000000000},\frac{244346392152949}{200000000000000}\, )
\end{eqnarray*}
and 

$$\tilde{H}_W(t_0+st+dt,120000,15)< \frac{77}{8000000000}$$

Therefore,

$$\dot{F}(t_0-st+dt,a_0+sa+da,b_0-sb)< -\frac{904301}{62500000000000}+ \frac{77}{8000000000} < 0.$$


$\dot{R}(t,a_0-sa+da,b_0+sb)>0$ for all $t\in[t_0+st-dt,t_0+st+dt]$, because $Z_W(t_0+st-dt,a_0-sa+da,b_0+sb,35000)$ equals to

$$ \left(\frac{49492383812687}{20000000000000},\frac{63630339}{100000000000000},\frac{945305585649929}{100000000000000},\frac{1366741}{2500000000000},\frac{24434631486741}{20000000000000}\right)$$

and 

$$\tilde{H}_W(t_0+st-dt,35000)< \frac{23713}{250000000000} $$

Therefore,

$$\dot{R}(t_0-st-dt,a_0+sa+da,b_0-sb)> \frac{1366741}{2500000000000}- \frac{23713}{250000000000}> 0$$


$\dot{R}(t,a_0+sa-da,b_0-sb)<0$ for all $t\in[t_0+st-dt,t_0+st+dt]$, because $Z_W(t_0+st+dt,a_0-sa-da,b_0+sb,35000)$ equals to

$$\left(\frac{123730949221087}{50000000000000},-\frac{33715619}{100000000000000},\frac{945305400490959}{100000000000000},-\frac{6891931}{25000000000000},\frac{1908956160269}{1562500000000}\right) $$

and 

$$\tilde{H}_W(t_0+st+dt,35000)<\frac{23713}{250000000000}$$

Therefore,

$$\dot{R}(t_0+st+dt,a_0+sa-da,b_0-sb)<-\frac{6891931}{25000000000000}+\frac{23713}{250000000000} < 0.$$

In order to show that $\Theta(t,a,b_0+sb)>\frac{7\pi}{18}$, we first noticed that $Z_W(t_0+st,a_0-sa,b_0+sb,35000)$ equals to

$$ \left(\frac{123730954376411}{50000000000000},\frac{3739331}{25000000000000},\frac{945305493070449}{100000000000000},\frac{1355097}{10000000000000},\frac{24434635169083}{20000000000000}\right)  $$

and 

$$\tilde{H}_W(t_0+st,35000)<   \frac{23713}{250000000000}.$$ 

Therefore, 

\begin{eqnarray*}\Theta(t_0-st,a_0+sa,b_0-sb)-\frac{7\pi}{18}&>&\frac{24434635169083}{20000000000000}-\frac{23713}{250000000000}-\frac{7\pi}{18}\\
& =&\frac{24434637066123}{20000000000000}-\frac{7 \pi }{18}
\end{eqnarray*}

Using the mean value theorem, and the bounds  $|\Theta_a|<\frac{483453}{1000000}$ and $|\dot{\Theta}|<\frac{27}{1000}$ found on corollaries \ref{ddot} and \ref{dFapos}, we obtain that  
for any $t\in [t_0+st-dt, t_0+st+dt]$ and $a\in[a_0-sa-da,a_0-sa+da]$

\begin{eqnarray*}
\Theta(t,a,b_0-sb)-\frac{7\pi}{18}& =& \Theta(t,a,b_0-sb)- \Theta(t,a_0-sa,b_0-sb)+\\
& &\Theta(t,a_0+sa,b_0-sb)- \Theta(t_0,a_0+sa,b_0-sb)+\\
& & \Theta(t_0,a_0+sa,b_0-sb)-\frac{7\pi}{18}\\
& >&\frac{24434637066123}{20000000000000}-\frac{7 \pi }{18}-\frac{483453}{1000000}\, da- \frac{27}{1000}\, dt\\
&>&0
\end{eqnarray*}

\end{proof}

%
%
\begin{lem}
Let $\delta_1=\frac{1017}{1250}$,  $\tilde{\delta}_1=\frac{8159}{10000}$, $\epsilon_1=\frac{2677}{5000}$, $\delta_2=\frac{16181}{10000}$,  $\tilde{\delta}_2=\frac{8359}{5000}$, $\epsilon_2=\frac{4549}{5000}$, $\tilde{\epsilon}_1=6(st+dt) $, $\tilde{\epsilon}_2=3(sa+da) $, $\mu_1=dt$ and $\mu_2=da$. If $\rho_1=\frac{\tilde{\epsilon}_1-\mu_1}{m_1}$, $\rho_2=\frac{\tilde{\epsilon}_2-\mu_2}{m_2}$, where $m_1=\frac{\epsilon_1(\tilde{\delta}_2+\epsilon_2)}{\delta_1\delta_2-\epsilon_1\epsilon_2}$ and $m_2=\frac{\epsilon_2(\tilde{\delta}_1+\epsilon_1)}{\delta_1\delta_2-\epsilon_1\epsilon_2}$, then, for every $t\in [t_0-\tilde{\epsilon}_1,t_0+\tilde{\epsilon}_1]$, $a\in [a_0-\tilde{\epsilon}_2,a_0+\tilde{\epsilon}_2]$ and $b\in [b_0-sb,b0+sb]$ we have that

$$\delta_1<|\ddot{F}(at,a,b)|<\tilde{\delta}_1,\quad |\dot{F}_a(t,a,b)|<\epsilon_1,\quad |\dot{F}_b(t,a,b)|<\epsilon_1 $$

and

$$ |\ddot{R}(t,a,b)|<\epsilon_2, \quad \delta_2< |\dot{R}_a(t,a,b)|<\tilde{\delta}_2, \quad |\dot{R}_b(t,a,b)|<\epsilon_2 $$

Moreover, $sb<\rho_1$ and $sb<\rho_2$.

\end{lem}
\begin{proof}
This is a direct computation. It uses Lemmas \ref{bds}, \ref{bdsb} and \ref{ddot}
\end{proof}

\begin{rem}
The previous lemma, along with Theorem \ref{iftt} implies Lemma \ref{l4}.
\end{rem}

\begin{rem}
Since we have shown the existence of a curve of initial conditions that provides reduced periodic functions, then we have not only shown the existence of the periodic solution with $\theta(t_0)=\frac{7\pi}{18}$ but we have shown the existence of infinitely many periodic solutions due to the fact that the function $\theta$ is not constant along this curve, and, on any open interval, there are infinitely many numbers of the form $\frac{p\pi}{q}$ with $p$ and $q$ whole numbers.
\end{rem}

\section{Dictionary of functions} \label{dic}

In this section we define the functions involved in the definition of the differential equations that we considered in this paper. The domain of  the variable  $x_1$ will be an interval containing possible values of the function $F(t,a,b)$. In the same way, the variable $x_2$ is related with the function $\dot{F}(t,a,b)$; the variable $x_3$ is related with the function $R(t,a,b)$; the variable $x_4$ is related with the function $\dot{R}(t,a,b)$;   the variable $x_5$ is related with the function $F_a(t,a,b)$; the variable $x_6$ is related with the function $\dot{F}_a(t,a,b)$;   the variable $x_7$ is related with the function $R_a(t,a,b)$; the variable $x_8$ is related with the function $\dot{R}_a(t,a,b)$;  the variable $x_{9}$ is related with the function $F_b(t,a,b)$; the variable $x_{10}$ is related with the function $\dot{F}_b(t,a,b)$; the variable $x_{11}$ is related with the function $R_b(t,a,b)$; the variable $x_{12}$ is related with the function $\dot{R}_b(t,a,b)$. As in the previous section, $s=\sqrt{4 x_1^2+x_3^2}$. For each one of these function $\text{B(}\phi _i\text{)= }\left\{\hbox{lb}_i,\hbox{ub}_i\right\}$ where $\hbox{lb}_i$ is a lower bound and $\hbox{ub}_i$ is a lower bound for the function $\phi_i=\phi_i(x_1,x_2,x_3,x_4,x_5,x_6,x_7,x_8,x_9,x_{10},x_{11},x_{12},a)$ on the set $V$, where 

$$V=I_1\times I_2 \times I_3\times I_4 \times I_5\times I_6 \times I_7\times I_8 \times I_9\times I_{10} \times I_{11} \times I_{12}  \times I_{13}$$

with

$$I_1=\left[-\frac{1}{100},\frac{62}{25}\right]\, I_2=\left[-\frac{1}{100},\frac{3}{2}\right] \, I_3=\left[\frac{189}{20},\frac{1001}{100}\right]\, I_4= \left[-\frac{33}{100},\frac{1}{100}\right] \, I_5=\left[-\frac{1}{100},\frac{31}{100}\right]$$

$$I_6=\left[-\frac{1}{100},\frac{12}{25}\right] \, I_7=\left[-\frac{1}{100},\frac{69}{25}\right] \, I_8= \left[-\frac{1}{100},\frac{42}{25}\right] \, I_9=\left[-\frac{1}{100},\frac{101}{50}\right]\, I_{10}=\left[\frac{1}{2},\frac{101}{100}\right] $$

$$ I_{11}=\left[-\frac{1}{100},\frac{81}{100}\right] \, I_{12}=\left[-\frac{1}{100},\frac{89}{100}\right]\, I_{13}=\left[\frac{43170106052787}{10000000000000},\frac{43170844652787}{10000000000000}\right]$$

We would like to point out that $I_{13}=[a_0-3(sa+da),a_0+3(sa+da)]$. 

\begin{rem} In order to obtain each bound $\text{B}(\phi_i)=\left\{\hbox{lb}_i,\hbox{ub}_i\right\}$, we first compute the minimum and maximum of the function, $mi$ and $ma$ respectively. Then we  define $\hbox{lb}_i=\frac{1}{10^k} \floor{10^k mi}$ and $\hbox{ub}_i=\frac{1}{10^k} \ceil{10^k ma}$, where $k$ is a positive integer. Usually $k$ is between 4 and 12 and it is chosen according to the precision that we want for the bound. The reason for using $\hbox{lb}_i,\hbox{ub}_i$ and not the minimum and maximum  is that $ma$ and $mi$ may have complicate expression in term of radicals and roots of polynomials.

\end{rem}

\begin{rem}

The bounds for each of the functions for which we are writing an expression in terms of the variables $x_1$, $\dots$, $x_{12}$ have been obtained by directly using the Lagrange multipliers method. For those functions $\phi_i$ that we are written in term of the previous $\phi_j$ and some $x_k$, we used the method of Lagrange multiplier by changing each $\phi_j$ by a variable $u_j$ and we used the fact that $u_j$ has bound given by $\text{B}(\phi_j)$. 
\end{rem}


\subsection{$\phi_1=\displaystyle{-\frac{400 x_1}{s^3}}$} Entry $2$ of the vector field $W$. $\text{B(}\phi _1\text{)= }\left\{-\frac{102003}{125000},\frac{237}{50000}\right\}$

\subsection{$\phi_2=\displaystyle{\frac{100 a^2}{x_3^3}-\frac{25}{x_3^2}-\frac{200 x_3}{s^3}}$} Entry $4$ of the vector field $W$. $\text{B(}\phi _2\text{)= }\left\{-\frac{387429}{1000000},\frac{373771}{1000000}\right\}$

\subsection{$\phi_3=\displaystyle{  \frac{10 a}{x_3^2} }$} Entry $5 $ of the vector field $W $. $\text{B(}\phi _3\text{)= }\left\{\frac{215419}{500000},\frac{483423}{1000000}\right\}$

\subsection{$\phi_4=\displaystyle{ -\frac{400 \left(x_3^2-8 x_1^2\right)}{s^5}  }$} Entry $(2,1) $ of the matrix $dW $. $\text{B(}\phi _4\text{)= }\left\{-\frac{94797}{200000},-\frac{115837}{1000000}\right\}$

\subsection{$\phi_5=\displaystyle{ \frac{1200 x_1 x_3}{s^5}   }$} Entry $(2,3) $ of the matrix $dW$. $\text{B(}\phi _5\text{)= }\left\{-\frac{301}{200000},\frac{101747}{500000}\right\}$

\subsection{$\phi_6=\displaystyle{  50 \left(-\frac{6 a^2}{x_3^4}+\frac{12 x_3^2}{s^5}-\frac{4}{s^3}+\frac{1}{x_3^3}\right) }$} Entry $(4,3) $ of the matrix $dW $. $\text{B(}\phi _6\text{)= }\left\{-\frac{83881}{200000},-\frac{108213}{1000000}\right\}$

\subsection{$\phi_7=\displaystyle{ -\frac{20 a}{x_3^3}  }$} Entry $(5,3) $ of the matrix $ dW $. $\text{B(}\phi _7\text{)= }\left\{-\frac{12789}{125000},-\frac{86081}{1000000}\right\}$

\subsection{$\phi_8=\displaystyle{-\frac{1200 \left(16 x_1^2 x_3-x_3^3\right)}{s^7}   }$} Part of the entry $(2,1) $ of the matrix $dW_1 $.
 $\text{B(}\phi _8\text{)= }\left\{-\frac{3273}{500000},\frac{18809}{125000}\right\}$

\subsection{$\phi_9=\displaystyle{  -\frac{4800 \left(8 x_1^3-3 x_1 x_3^2\right)}{s^7} }$} Part of the entry $(2,1) $ of the matrix $dW_1 $. 
$\text{B(}\phi _9\text{)= }\left\{-\frac{1911}{1000000},\frac{193837}{1000000}\right\}$

\subsection{$\phi_{10}=\displaystyle{x_4 \phi_8+x_2\phi_9   }$} Entry $ (2,1)$ of the matrix $dW_1 $.
 $\text{B(}\phi _{10}\text{)= }\left\{-\frac{52523}{1000000},\frac{73229}{250000}\right\}$

\subsection{$\phi_{11}=\displaystyle{ \frac{4800 \left(x_1^3-x_1 x_3^2\right)}{s^7}  }$} Part of the entry $(2,3) $ of the matrix $dW_1 $. 
$\text{B(}\phi _{11}\text{)= }\left\{-\frac{34393}{500000},\frac{637}{1000000}\right\}$

\subsection{$\phi_{12}=\displaystyle{  x_4 \phi_{11}+x_2\phi_8  }$} Entry $(2,3) $ of the matrix $ dW_1 $.
 $\text{B(}\phi _{12}\text{)= }\left\{-\frac{10507}{1000000},\frac{31051}{125000}\right\}$

\subsection{$\phi_{13}=\displaystyle{  \frac{1200 a^2}{x_3^5}-\frac{3000 x_3^3}{s^7}+\frac{1800 x_3}{s^5}-\frac{150}{x_3^4} }$} Part of the Entry $ (4,3)$ of the matrix $ dW_1 $. $\text{B(}\phi _{13}\text{)= }\left\{\frac{44031}{500000},\frac{9611}{40000}\right\}$

\subsection{$\phi_{14}=\displaystyle{  x_4 \phi_{13}+2 x_2\phi_{11}  }$} Entry $(4,3) $ of the matrix $dW_1 $. $\text{B(}\phi _{14}\text{)= }\left\{-\frac{285649}{1000000},\frac{2157}{500000}\right\}$

\subsection{$\phi_{15}=\displaystyle{ \frac{60 a x_4}{x_3^4}  }$} Entry $ (5,3) $ of the matrix $ dW_1$. $\text{B(}\phi _{15}\text{)= }\left\{-\frac{10719}{1000000},\frac{13}{40000}\right\}$

\subsection{$\phi_{16}=\displaystyle{  \frac{400 \left(8 x_5 x_1^2+3 x_3 x_7 x_1-x_3^2 x_5\right)}{s^5} }$} Entry $6 $ of the vector field $G $. $\text{B(}\phi _{16}\text{)= }\left\{-\frac{75543}{500000},\frac{563009}{1000000}\right\}$

\subsection{$\phi_{17}=\displaystyle{  -10 \phi_7+2  x_5 \phi_{5}+ x_7\phi_{6}  }$} Entry $8 $ of the vector field $G $. 
$\text{B(}\phi _{17}\text{)= }\left\{-\frac{150409}{500000},\frac{1153481}{1000000}\right\}$

\subsection{$\phi_{18}=\displaystyle{  \frac{10 \left(x_3-2 a x_7\right)}{x_3^3} }$} Entry $9 $ of the vector field $ G $.
 $\text{B(}\phi _{18}\text{)= }\left\{-\frac{85201}{500000},\frac{113003}{1000000}\right\}$

\subsection{$\phi_{19}=\displaystyle{ x_5 \phi_9+x_7 \phi_8  }$} Entry $(6,1) $ of the matrix $dG $.
 $\text{B(}\phi _{19}\text{)= }\left\{-\frac{10003}{500000},\frac{475393}{1000000}\right\}$

\subsection{$\phi_{20}=\displaystyle{ x_5 \phi_8+x_7 \phi_{11}   }$} Entry $(6,3) $ of the matrix $dG $. 
$\text{B(}\phi _{20}\text{)= }\left\{-\frac{191879}{1000000},\frac{9681}{200000}\right\}$

\subsection{$\phi_{21}=\displaystyle{  -\frac{600 a}{x_3^4} }$} Part of the entry $(8,3) $ of the matrix $dG $. 
$\text{B(}\phi _{21}\text{)= }\left\{-\frac{324799}{1000000},-\frac{257987}{1000000}\right\}$

\subsection{$\phi_{22}=\displaystyle{  \phi_{21}+2 x_5 \phi_{1}+x_7 \phi_{13}    }$} Entry $ (8,3) $ of the matrix $dG $.
 $\text{B(}\phi _{22}\text{)= }\left\{-\frac{147363}{1000000},\frac{289227}{500000}\right\}$

\subsection{$\phi_{23}=\displaystyle{  \frac{60 a x_7}{x_3^4}-\frac{20}{x_3^3}    }$} Entry $ (9,3)$ of the matrix $dG $.
 $\text{B(}\phi _{23}\text{)= }\left\{-\frac{961}{40000},\frac{32973}{500000}\right\}$

\subsection{$\phi_{24}=\displaystyle{ x_8 \phi_{8}+x_6 \phi_{9}  }$} Part of the entry $(6,1) $ of the matrix $ dG_1 $. $\text{B(}\phi _{24}\text{)= }\left\{-\frac{1617}{125000},\frac{69167}{200000}\right\}$


\subsection{$\phi_{25}=\displaystyle{  -\frac{4800 \left(16 x_1^4-27 x_3^2 x_1^2+x_3^4\right)}{s^9}   }$} Part of the entry $(6,1) $ of the matrix $ dG_1 $. 
$\text{B(}\phi _{25}\text{)= }\left\{-\frac{15923}{250000},\frac{8349}{500000}\right\}$

\subsection{$\phi_{26}=\displaystyle{   \frac{24000 \left(16 x_1^3 x_3-3 x_1 x_3^3\right)}{s^9}   }$} Part of the entry $(6,1) $ of the matrix $ dG_1 $. $\text{B(}\phi _{26}\text{)= }\left\{-\frac{85577}{1000000},\frac{1011}{1000000}\right\}$

\subsection{$\phi_{27}=\displaystyle{   x_4 \phi_{25}+x_2 \phi_{26}  }$} Part of the entry $(6,1) $ of the matrix $ dG_1 $.
 $\text{B(}\phi _{27}\text{)= }\left\{-\frac{33469}{250000},\frac{4507}{200000}\right\}$

\subsection{$\phi_{28}=\displaystyle{  \frac{4800 \left(128 x_1^4-96 x_3^2 x_1^2+3 x_3^4\right)}{s^9}   }$} Part of the entry $(6,1) $ of the matrix $ dG_1 $.
 $\text{B(}\phi _{28}\text{)= }\left\{-\frac{64021}{1000000},\frac{7643}{40000}\right\}$

\subsection{$\phi_{29}=\displaystyle{  x_2 \phi_{28}+x_4 \phi_{26}    }$} Part of the entry $(6,1) $ of the matrix $ dG_1 $.
 $\text{B(}\phi _{29}\text{)= }\left\{-\frac{12111}{125000},\frac{314853}{1000000}\right\}$

\subsection{$\phi_{30}=\displaystyle{   \phi_{24} +  x_7 \phi_{27}+x_5 \phi_{29}   }$} Entry $(6,1) $ of the matrix $ dG_1 $. 
$\text{B(}\phi _{30}\text{)= }\left\{-\frac{41247}{100000},\frac{505637}{1000000}\right\}$


\subsection{$\phi_{31}=\displaystyle{ x_6 \phi_8+x_8 \phi_{11} }$} Part of the entry $(6,3) $ of the matrix $ dG_1 $. 
$\text{B(}\phi _{31}\text{)= }\left\{-\frac{118703}{1000000},\frac{73297}{1000000}\right\}$

\subsection{$\phi_{32}=\displaystyle{  -\frac{24000 \left(3 x_1^3 x_3-x_1 x_3^3\right)}{s^9}  }$} Part of the entry $(6,3) $ of the matrix $ dG_1 $.
 $\text{B(}\phi _{32}\text{)= }\left\{-\frac{337}{1000000},\frac{6067}{200000}\right\}$.

\subsection{$\phi_{33}=\displaystyle{ x_2 \phi_{25}+x_4 \phi_{32}    }$} Part of the entry $(6,3) $ of the matrix $ dG_1 $.
 $\text{B(}\phi _{33}\text{)= }\left\{-\frac{105549}{1000000},\frac{25351}{1000000}\right\}$

\subsection{$\phi_{34}=\displaystyle{  \phi_{31} +  x_7 \phi_{33}+x_5 \phi_{27}    }$} Entry $(6,3) $ of the matrix $ dG_1 $. 
$\text{B(}\phi _{34}\text{)= }\left\{-\frac{1411}{3125},\frac{37563}{250000}\right\}$

\subsection{$\phi_{35}=\displaystyle{  \frac{2400 a x_4}{x_3^5}  }$} Part of the entry $(8,3) $ of the matrix $ dG_1 $. 
$\text{B(}\phi _{35}\text{)= }\left\{-\frac{45369}{1000000},\frac{11}{8000}\right\}$

\subsection{$\phi_{36}=\displaystyle{   \phi_{35} +  x_8 \phi_{13}+ 2 x_6 \phi_{11}    }$} Part of the entry $(8,3) $ of  $ dG_1 $. 
$\text{B(}\phi _{36}\text{)= }\left\{-\frac{113807}{1000000},\frac{406413}{1000000}\right\}$

\subsection{$\phi_{37}=\displaystyle{  600 \left(-\frac{10 a^2}{x_3^6}+\frac{8 x_3^4}{s^9}-\frac{108 x_1^2 x_3^2}{s^9}+\frac{12 x_1^2}{s^7}+\frac{1}{x_3^5}\right)  }$} Part of the entry $(8,3) $ of the matrix $ dG_1 $. 
$\text{B(}\phi _{37}\text{)= }\left\{-\frac{579}{4000},-\frac{2871}{50000}\right\} $

\subsection{$\phi_{38}=\displaystyle{     x_4 \phi_{37}+ 2 x_2 \phi_{32}    }$} Part of the entry $(8,3) $ of the matrix $ dG_1 $. 
$\text{B(}\phi _{38}\text{)= }\left\{-\frac{2459}{1000000},\frac{138773}{1000000}\right\}$

\subsection{$\phi_{39}=\displaystyle{    \phi_{36} + 2 x_5 \phi_{33}+ x_7 \phi_{38}   }$} Entry $(8,3) $ of the matrix $ dG_1 $. 
$\text{B(}\phi _{39}\text{)= }\left\{-\frac{117597}{1000000},\frac{400107}{1000000}\right\}$

\subsection{$\phi_{40}=\displaystyle{   \frac{60 \left(x_3 \left(a x_8+x_4\right)-4 a x_4 x_7\right)}{x_3^5}  }$} Entry $(9,3) $ of the matrix $ dG_1 $.
$\text{B(}\phi _{40}\text{)= }\left\{-\frac{2853}{1000000},\frac{32303}{500000}\right\}$

\subsection{$\phi_{41}=\displaystyle{ \frac{400 \left(8 x_9 x_1^2+3 x_3 x_{11} x_1-x_3^2 x_9\right)}{s^5}  }$} Entry $6 $ of the vector field $ U $. 
$\text{B(}\phi _{41}\text{)= }\left\{-\frac{19173}{20000},\frac{83109}{500000}\right\}$

\subsection{$\phi_{42}=\displaystyle{ 2 x_9 \phi_5+x_{11} \phi_6  }$} Entry $8 $ of the vector field $ U $. 
$ \text{B(}\phi _{42}\text{)= }\left\{-\frac{19549}{500000},\frac{205701}{250000}\right\} $

\subsection{$\phi_{43}=\displaystyle{ x_{11} \phi_8+x_9 \phi_9  }$} Entry $(6,1) $ of the matrix $ dU $. 
$ \text{B(}\phi _{43}\text{)= }\left\{-\frac{9163}{1000000},\frac{256717}{500000}\right\} $

\subsection{$\phi_{44}=\displaystyle{ x_{11} \phi_{11}+x_9\phi_8  }$} Entry $(6,3) $ of the matrix $ dU $. 
$ \text{B(}\phi _{44}\text{)= }\left\{-\frac{3447}{50000},\frac{152321}{500000}\right\}  $

\subsection{$\phi_{45}=\displaystyle{ 2 x_9 \phi_{11}+x_{11} \phi_{13}  }$} Entry $(8,3) $ of the matrix  $ dU $. 
$ \text{B(}\phi _{45}\text{)= }\left\{-\frac{280299}{1000000},\frac{197197}{1000000}\right\} $

\subsection{$\phi_{46}=\displaystyle{ x_{12} \phi_8 + x_{10} \phi_9  }$} Part of the entry $(6,1) $ of the vector field $ dU_1 $. 
$ \text{B(}\phi _{46}\text{)= }\left\{-\frac{7757}{1000000},\frac{10303}{31250}\right\} $

\subsection{$\phi_{47}=\displaystyle{  x_4 \phi_{26} + x_2 \phi_{28}  }$} Part of the entry $(6,1) $ of the matrix  $ dU_1 $. 
$\text{B(}\phi _{47}\text{)= }\left\{-\frac{12111}{125000},\frac{314853}{1000000}\right\}  $

\subsection{$\phi_{48}=\displaystyle{  x_{11} \phi_{27}+x_9 \phi_{47}+ \phi_{46}  }$} Entry $(6,1) $ of the matrix  $ dU_1 $. 
$ \text{B(}\phi _{48}\text{)= }\left\{-\frac{203233}{500000},\frac{35511}{62500}\right\} $

\subsection{$\phi_{49}=\displaystyle{  x_{12} \phi_{11}+x_{10} \phi_8  }$} Part of the entry $(6,3) $ of the matrix  $ dU_1 $. 
$ \text{B(}\phi _{49}\text{)= }\left\{-\frac{67831}{1000000},\frac{30533}{200000}\right\} $

\subsection{$\phi_{50}=\displaystyle{  x_9 \phi_{27}+x_{11} \phi_{33} + \phi_{49}  }$} Entry $(6,3) $ of the matrix  $ dU_1 $. 
$ \text{B(}\phi _{50}\text{)= }\left\{-\frac{105939}{250000},\frac{218721}{1000000}\right\}  $

\subsection{$\phi_{51}=\displaystyle{  \frac{1200 a^2}{x_3^5}-\frac{1200 x_3^5}{s^9}+\frac{2400 x_1^2 x_3^3}{s^9}+\frac{28800 x_1^4 x_3}{s^9}-\frac{150}{x_3^4}  }$} Part of the  entry $(8,3) $ of the matrix  $ dU_1 $. 
$  \text{B(}\phi _{51}\text{)= }\left\{\frac{44031}{500000},\frac{9611}{40000}\right\} $

\subsection{$\phi_{52}=\displaystyle{  2 x_{10} \phi_{11} + x_{12} \phi_{51}  }$} Part of the entry $(8,3) $ of the matrix  $ dU_1 $. 
$ \text{B(}\phi _{52}\text{)= }\left\{-\frac{141351}{1000000},\frac{53783}{250000}\right\} $

\subsection{$\phi_{53}=\displaystyle{  2 x_9 \phi_{33} + x_{11} \phi_{38}+ \phi_{52}  }$} Entry $(8,3) $ of the matrix  $ dU_1 $. 
$ \text{B(}\phi _{53}\text{)= }\left\{-\frac{569761}{1000000},\frac{429957}{1000000}\right\} $

\subsection{$\phi_{54}=\displaystyle{ \frac{100 (a-a_0) (a+a_0)}{x_3^3} }$} Entry 4  of $W-W_0$. 
$\text{B(}\phi _{54}\text{)= }\left\{-\frac{19}{500000},\frac{19}{500000}\right\}$

\subsection{$\phi_{55}=\displaystyle{\frac{10 (a-a_0)}{x_3^2} }$} Entry 5  of $W-W_0$. 
$\text{B(}\phi _{55}\text{)= }\left\{-\frac{1}{200000},\frac{1}{200000}\right\}$

\subsection{$\phi_{56}=\displaystyle { -\frac{100 (a-a_0) \left(3 a x_7+3 a_0 x_7-2 x_3\right)}{x_3^4} }$} Entry 8 of $G-G_0$. 
$ \text{B(}\phi _{56}\text{)= }\left\{-\frac{1}{40000},\frac{1}{40000}\right\} $

\subsection{$\phi_{57}=\displaystyle{ -\frac{20 x_7 (a-a_0)}{x_3^3}  }$} Entry 9 of $G-G_0$. 
$\text{B(}\phi _{57}\text{)= }\left\{-\frac{3}{1000000},\frac{3}{1000000}\right\}$

\subsection{$\phi_{58}=\displaystyle{ -\frac{300 x_{11} (a-a_0) (a+a_0)}{x_3^4} }$} Entry 8 of $U-U_0$. 
$\text{B(}\phi _{57}\text{)= }\left\{-\frac{1}{1000000},\frac{1}{1000000}\right\}$


\end{document}